\documentclass[10pt]{article}
\usepackage[utf8]{inputenc}
\usepackage{arxiv}

\usepackage{epsfig}
\usepackage{array}
\usepackage{amssymb,amsmath,amsfonts,moresize}
\usepackage{graphicx,epstopdf} 
\usepackage{color,xcolor}
\usepackage{booktabs}
\usepackage{multirow}
\usepackage{amstext} 
\usepackage{array}   
\usepackage{amsthm}
\usepackage{authblk}
\newcolumntype{C}{>{$}c<{$}}
 \newtheorem{theorem}{Theorem}
 \newtheorem{definition}{Definition}
 \newtheorem{proposition}{Proposition}
 
\newtheorem{remark}{Remark}

\providecommand{\keywords}[1]
{
  \small	
  \textbf{\textit{Keywords---}} #1
}

\title{Approximate  well-balanced WENO finite difference   schemes  using a global-flux quadrature method \\with multi-step ODE integrator weights}
\author{M.Kazolea$^1$, C. Par\'es$^2$, M. Ricchiuto}
\affil[1]{INRIA, Univ. Bordeaux, CNRS, Bordeaux INP, IMB, UMR 5251, 200 Avenue de la Vieille Tour, 33405, Talence cedex, France}
\affil[2]{ University of Málaga, Spain}
\date{ }

\begin{document}

\maketitle
\begin{abstract}
In this work, high-order discrete well-balanced methods for one-dimensional hyperbolic systems of balance laws are proposed. We aim to construct a method whose discrete steady states correspond to solutions of arbitrary high-order ODE integrators. However, this property is embedded directly into the scheme, eliminating the need to apply the ODE integrator explicitly to solve the local Cauchy problem.
To achieve this, we employ a WENO finite difference framework and apply WENO reconstruction to a global flux assembled nodewise as the sum of the physical flux and a source primitive. The novel idea is to compute the source primitive using high-order multi-step ODE methods applied on the finite difference grid. This approach provides a locally well-balanced splitting of the source integral, with weights derived from the ODE integrator. By construction, the discrete solutions of the proposed schemes align with those of the underlying ODE integrator.
The proposed methods employ WENO flux reconstructions of varying orders, combined with multi-step ODE methods of up to order 8, achieving steady-state accuracy determined solely by the ODE method's consistency. Numerical experiments using scalar balance laws and shallow water equations confirm that the methods achieve optimal convergence for time-dependent solutions and significant error reduction for steady-state solutions.
\end{abstract}

\section{Introduction}
This paper is devoted to the solution of  one-dimensional systems of balance laws of the form
\begin{equation}\label{sbl}
U_t + F(U)_x = S(U) H_x + s(U,x),
\end{equation}
where $U(x,t)$ takes value in $\Omega \subset \mathbb{R}^N$, $F: \Omega  \to \mathbb{R}^N$ is the flux function;
$S: \Omega \to \mathbb{R}^N$; and $H$ is a known function from $\mathbb{R} \to \mathbb{R}$ (possibly the identity function $H(x) = x$).  The last term $s: \Omega\times\mathbb{R} \to \mathbb{R}^N$ accounts for extra  point-wise  sources.
Let us assume initially that $H$ is continuous and $s$ is not singular. The system is supposed to be hyperbolic, i.e. the Jacobian $J(U)$ of the flux function is assumed to have $N$ different real eigenvalues. 

Prototypes of scalar and systems of balance laws of the form \eqref{sbl} are:

\begin{itemize}
\item Burgers' equations with source term which corresponds to $N = 1$ and
\begin{equation}\label{Burgers}
F(U) = \frac{U^2}{2}, \quad S(U) = U^p, \quad s(U,x) \equiv 0.
\end{equation}

\item  The shallow water system  which corresponds to $N = 2$ and
\begin{equation}\label{sw}
U = \left[ \begin{array}{c} h \\ q \end{array} \right], \quad F(U) = \left[ \begin{array}{c} q \\\displaystyle \frac{q^2}{h} + \frac{g}{2}h^2 \end{array} \right], 
\quad S(U) = \left[ \begin{array}{c} 0 \\ gh \end{array} \right]. 
\end{equation}
The variable $x$ makes reference to the axis of the channel and $t$ is time; $q(x,t)$ and $h(x,t)$ represent the mass-flow and the thickness,
respectively; $g$, the acceleration due to gravity;  $H(x)$, the depth measured from a fixed level of reference;  $q(x,t)=h(x,t)u(x,t),$ with $u$ the depth averaged horizontal velocity. If friction is also accounted for then we have
\begin{equation}\label{sw1}
s(U,x) = \left[ \begin{array}{c} 0 \\ -\kappa(U,x) q \end{array} \right] 
\end{equation}
with $\kappa(U,x)\ge 0$ a  friction coefficient depending usually on  hydrodynamic assumptions and on the  bottom rugosity. Other effects can also be included, as e.g. wind forcing or variable atmospheric pressure, but will not be considered here. 

\end{itemize}

We consider semi-discrete finite-difference numerical methods of the form
\begin{equation}\label{sdmeth}
\frac{d U_i}{dt} + \frac{1}{\Delta x}   \left({F}_{i + 1/2} - {F}_{i-1/2} \right) + S_i=  0,
\end{equation}
Here $U_i(t)$ represents an approximation of $U(x_i, t)$ where $x_0, \dots, x_N$ are the points of a uniform mesh of the computational interval $[a, b]$ and $\Delta x$ is the space step; $F_{i + 1/2}$ is the numerical flux; and $S_i$ the numerical source term. 
We focus here on the development of high-order well-balanced methods for systems of the form \eqref{sbl}, i.e. schemes that preserve in some sense all or some of the stationary solutions of the system, i.e. 
functions $U^*(x)$ that satisfy
$$
F(U^*)_x = S(U^*) H_x + s(U^*,x).$$
For the prototypes introduced above, the stationary solutions are as follows:
\begin{itemize}

\item For Burgers' equations with source term \eqref{sbl}-\eqref{Burgers} the stationary solutions satisfy the ODE
$$
U_x = U^{p-1} H_x$$
whose solutions are given by
\begin{equation}\label{ssBurgers2}
U e^{-H} = C
\end{equation}
if $p = 2$ and
\begin{equation}\label{ssBurgersp}
\frac{U^{2-p}}{2-p} - H = C
\end{equation}
otherwise. Here, $C$ represents an arbitrary constant.

\item  For the shallow water system \eqref{sbl}-\eqref{sw}, the stationary solutions satisfy in absence of friction
\begin{equation}\label{ssSW}
q = C_1, \quad \frac{1}{2}\frac{q^2}{h^2} + gh - gH = C_2,
\end{equation} 
where $C_i$, $i = 1,2$ are arbitrary constants. If $C_1 = 0$ water at rest or lake at rest stationary solutions
\begin{equation}\label{sswarSW}
q = 0, \quad \eta  = C_2,
\end{equation} 
are obtained, where $\eta = h - H$ is the free surface elevation.

In presence of friction, the second in \eqref{ssSW} must be replaced by
the ODE
\begin{equation}\label{ssSW1}
- C^2_1 h^{-2} h_x + gh(h_x -H_x)  + \kappa(h,C_1,x) C_1 =0 
\end{equation}
which can be integrated exactly in some particular cases (cf section \S5.3.3).

\end{itemize}

Well-balanced methods are expected to preserve all the stationary solutions or at least a relevant family of them. They are also expected to correctly handle the waves produced by small perturbations of an equilibrium given by a steady state solution. Different definitions of well-balanced methods can be found in the literature according to the set of stationary solutions that is preserved and to the accuracy with which they are preserved.  The following definitions will be adopted in this work:

\begin{definition} 
The set of point-values $\{ U_i \}_{i=0}^N$ is said to be a discrete stationary solution if it is an equilibrium for the ODE system \eqref{sdmeth}, i.e. if 
$$
 \frac{1}{\Delta x}   \left({F}_{i + 1/2} - {F}_{i-1/2} \right) + S_i=  0.
 $$
\end{definition}

\begin{definition} The numerical method \eqref{sdmeth} is said to be exactly well-balanced for a stationary solution $U^*$ of \eqref{sbl} if the set of its point values
$\{ U^*(x_i) \}$ is a discrete stationary solution.
\end{definition}

\begin{definition} The numerical method \eqref{sdmeth} is said to be well-balanced for a stationary solution $U^*$ of \eqref{sbl} if, for all $\Delta x$, it is possible to find a discrete stationary solution $\{ U^*_i \}$ that approximates $\{U^*(x_i) \}$ with an order of accuracy greater than that of the method \eqref{sdmeth}.
\end{definition}

\begin{definition} The numerical method \eqref{sdmeth} is said to be fully exactly well-balanced (resp. fully well-balanced) if it is exactly well-balanced (resp. well-balanced) for every stationary solution.
\end{definition}

The development of well-balanced numerical methods is a very active front of research. Finite volume methods based on approximate Riemann solvers that preserve exactly a family of stationary solutions were introduced by Roe \cite{Roe87}, and then  by V\'azquez and Berm\'udez in \cite{Bermudez-Vazquez1994}: in the later numerical methods that exactly preserve water-at-rest solutions were introduced for the shallow-water equations. This well-balance property was called there $C$-property. A different strategy based on the use of standard Riemann problems for the homogeneous problem and the so-called hydrostatic reconstruction to discretize the source term was introduced in  \cite{audusse2004fast} and generalized in \cite{GHR2007} in the framework of path-conservative numerical methods (see also \cite{castro2017well}). Fully well-balanced numerical methods based on Riemann solvers were developed by  Greenberg and Leroux \cite{doi:10.1137/0733001}, and then \cite{GALLOUET2003479}, and more recently \cite{berthon2016fully}
        and subsequent works (e.g. \cite{MICHELDANSAC2017115} or \cite{berthon:hal-04620125} and references therein).

Strategies based on the computation of local equilibria at every cell and every time step have been followed in different works to obtain fully well-balanced methods. In the context of finite volume methods, such methods were built in \cite{CGLP08} for semilinear balance laws by using the so-called well-balanced reconstruction operators: after having computed the local equilibrium, a standard reconstruction operator is applied to the fluctuations of the cell values in the stencil with respect to the local steady state; the reconstruction is then given by the sum of the equilibrium and the reconstructed fluctuations. Since then, this technique has been extended to general systems of balance laws: see     \cite{Castro-CP2020} and the references therein. In particular, it was applied to the shallow-water equations in \cite{CLP13}. Different applications and extensions of this technique can be found in the literature: in the context of Euler equations with gravity, a similar technique was used in \cite{berberich2020} to develop methods that are well-balanced for all the hydrostatic stationary solutions; a strategy based on the use of approximate local equilibria (computed by applying a RK-collocation solver to the ODE system satisfied by the stationary solutions) was introduced in \cite{gomez2021collocation} that allows one to use well-balanced reconstructions in problems in which the expression of the stationary solutions is not available neither explicitly nor implicitly; in  \cite{GCD18} this technique was applied to obtain well-balanced arbitrary-lagrangian-eulerian finite volume schemes for the Euler equations, etc. 

The strategy based on well-balanced reconstructions was extended to finite-difference methods  in  \cite{Castro-CP2020}, where WENO reconstructions were applied to flux fluctuations. These methods are related to those based on the subtraction of the equations satisfied by a given stationary solution to be preserved (see \cite{BERBERICH2021104858} for instance).

An alternative strategy to build high-order well-balanced methods is based on the reconstruction in equilibrium variables: see \cite{Russo2013}, \cite{Kurganov2016} and more recently \cite{Shu24} where a fifth-order fully well-balanced finite-difference method has been introduced. 

Another way to construct well balanced approximations is to  use the primitive of the evolution operator  $G=F-\int (S H_x +s)$   as equilibrium variables.  This method is equivalent to recasting \eqref{sbl} in a pseudo-conservation law form, since now the time evolution of the unknown can be formally expressed as the derivative  of the global flux $G$. Methods based on the reconstruction of global fluxes   have been introduced initially to devise TVD schemes for balance laws \cite{GASCON2001261,CASELLES200916}. Their use has seen
a recent increased interest to produce fully well balanced methods 
         \cite{Kurganov2019,Ciallella2023} which embed consistency 
         with some integrated form of the steady state ODE. A clear characterization of this consistency has been provided in \cite{MANTRI2024112673},  in which, 
         for finite element  based schemes, the authors show that the global flux formulation allows
         to recover specific continuous collocation discretizations of the steady ODE.
         This has the advantage of giving a clear definition of the discrete steady state,   with   rigorous (super)consistency estimates.  The formulation in the last reference shows that the iterative computation of the global flux on the mesh is not always necessary. For this reason the approach is rater referred to as global flux \emph{quadrature}, as it boils down
         to a non-trivial quadrature of the source terms. 
         Global flux techniques allow to embed the full evolution operator and have been shown to have many similarities to residual based techniques \cite{Abgrall2022,micalizzi24}. An approach to generalize this notions to multiple dimensions has also been proposed in \cite{brt25}. Global flux  quadrature methods  are not exactly well balanced in general, and fall in the class of methods of Definitions  3 and 4.

Coming back to  \cite{ParesPulido}, the family of well-balanced high-order methods for \eqref{sbl}  introduced there for the case $s=0$ write as follows:
\begin{equation}\label{wbsdmeth}
\frac{d U_i}{dt} + \frac{1}{\Delta x}   \left( \widehat{\mathcal{F}}_{i,i + 1/2} - \widehat{\mathcal{F}}_{i,i-1/2} \right)  =  0,
\end{equation}
where 
the \lq\lq numerical fluxes\rq\rq $\widehat{\mathcal{F}}_{i, i\pm 1/2}$ are computed as follows:
\begin{enumerate}

\item Look for the {solution} $U^*_i (x)$ of the Cauchy problem 
 \begin{equation}\label{Cauchy}
 \begin{cases}
      & \displaystyle F(U^*_i)_x = S(U^*_i)H_x + s(U^*,x), \\[0.5em]
      & \displaystyle U^*_i(x_i) = U_i.
 \end{cases}
 \end{equation}

\item Define
\begin{equation}\label{mod-flux1} 
\mathcal{F}_j = F(U_j) - F(U^*_i(x_j)), \quad j = i-1-r,\dots, i+s
\end{equation}

\item  Compute
\begin{eqnarray}
\widehat{\mathcal{F}}_{i,i+1/2} & =&  {\mathcal{R}}(\mathcal{F}_{i-r}, \dots, \mathcal{F}_{i+s}), \label{fi+1/2-1}\\
\widehat{\mathcal{F}}_{i,i-1/2} & = & {\mathcal{R}}(\mathcal{F}_{i-1-r}, \dots, \mathcal{F}_{i-1+s}), \label{fi-1/2-1}
\end{eqnarray}
where $\mathcal{R}$ represents a flux-reconstruction operator that gives high-order approximation of the value at the intercells of a function from the values at the points of the stencil $\mathcal{S}_i = \{ x_{i-r}, \dots, x_{i + s} \} $. 

\end{enumerate}

Observe that, \eqref{wbsdmeth} can be written in the form \eqref{sdmeth} with
\begin{equation}\label{numflux}
{F}_{i+1/2} =   {\mathcal{R}}({F}(U_{i-r}), \dots, {F}(U_{i+s})),
\end{equation}
and
\begin{equation} \label{source1}
S_i  =  \frac{1}{\Delta x}\left({\mathcal{R}}\Bigl(F(U_i^*(x_{i-r})), \dots, F(U_i^*(x_{i+s}) )\Bigr)-  {\mathcal{R}}\Bigl(F(U^* _i(x_{i-1-r})), \dots, F(U^*_i(x_{i-1+s}))\Bigr)  \right),
\end{equation}
if the reconstruction operator is such that
\begin{equation}\label{Rlinear}
{\mathcal{R}}(\mathcal{F}_{i-r}, \dots, \mathcal{F}_{i+s})
= {\mathcal{R}}({F}_{i-r}, \dots, {F}_{i+s}) - {\mathcal{R}}\Bigl(F(U_i^*(x_{i-r})), \dots, F(U_i^*(x_{i+s}) )\Bigr).
\end{equation}
These methods are fully exactly well-balanced. Their main drawback is that a local stationary solution has to be computed at every point and every time step. This  may be costly or even  impossible when the analytic expression in explicit or implicit form of the stationary solutions is not available. The method is for example not applicable in general to flows with friction. In these cases, an alternative is to apply an RK-collocation method to the Cauchy problem \eqref{Cauchy}, like in \cite{gomez2021collocation} to obtain approximations $U^*_{i,j} \approx U^*_i(x,j) $ and then define
$$\mathcal{F}_j = F(U_j) - F(U^*_{i,j}), \quad j = i-1-r,\dots, i+s.$$
The numerical methods are then fully well-balanced but not exactly: what is preserved now is not the set of point values of stationary solutions but their approximations using the selected RK collocation method.

The goal of this paper is to obtain well-balanced high-order methods that do not require  necessarily the computation of stationary solutions neither exact nor approximate at every mesh point in every time step. We do however require
a rigorous characterization of the discrete steady state.
To do this,  we start from the global flux approach. The PDE system is thus written in the form
$$
U_t + F(U)_x - R_x = 0,
$$
where
$$
R = \int_{a}^x \big( S(U)H_x +  s(U,x)\big)\, dx$$
for some $a \in \mathbb{R}$; then, the flux reconstruction operator is applied to reconstruct $F(U) - R$. 

In \cite{Ciallella2023}, a quadrature formula with points in the cell was considered to compute $R$ and WENO reconstructions were used to approximate the solution at the quadrature points. The key point here is that  a quadrature formula
$$
\int_{x_{j}}^{x_{j + 1}} S(U)H_x \,dx \approx \Delta x \sum_{m=0}^s \beta_m \big[ S(U_{j +1 -s + m })H_x(x_{j + 1 - s +   m})   + s_{j +1 -s + m }  \big]$$
using the approximations of the solution at the left will be considered instead. 
Moreover,  in the spirit of \cite{MANTRI2024112673},  we build some inherent consistency with well 
knwon high order ODE integrators. To this end,
the coefficients $\beta_m$ of the formula will be those of a multistep ODE solver and it will be shown that the numerical methods are well-balanced in the sense that they preserve the approximations 
of a stationary solution obtained by applying the multistep method to \eqref{Cauchy} in an adequate form. This has the neat advantage that, if necessary,  with one sweep of the ODE solver
one can obtain the discrete steady state to use as well prepared initial data.
The methods are  extended to the approximations of singular sources, and applied to the scalar Burger equation, and to the 
 shallow water equations with bathymetry, and friction.

The organization of the paper is as follows: in Section 2, the general form of the high-order finite-difference methods based on global-flux quadrature are introduced. Then, the discrete stationary solutions are studied and the well-balanced property is stated and shown. For the sake of completeness, the expression of the WENO reconstruction operators used here is given to conclude the section. Section 3 is devoted to enhancements of the methods: their extension to problems with singular source terms and their modification to satisfy the $C$-property, i.e. to exactly preserve water-at-rest solution for the shallow-water equations. The implementation details are given in Section 4 and the numerical results are shown in Section 5. There, the numerical methods are applied to the two prototypes equations and also to the shallow-water equations with friction to check their accuracy and well-balance properties. Finally conclusions are drawn.

\section{Well-balanced scheme using global flux quadrature} \label{S:methods}

\subsection{Global-flux approach}
High-order finite-difference methods for \eqref{sbl} of the form 
\begin{equation}\label{wbsdmeth2}
\frac{d U_i}{dt} + \frac{1}{\Delta x}   \left( \widehat{\mathcal{F}}_{i + 1/2} - \widehat{\mathcal{F}}_{i-1/2} \right)  =  0
\end{equation}
will be considered. 
The algorithm to compute the generalized numerical fluxes, based on the global-flux approach, is as follows:

\begin{enumerate}

\item Compute approximations
$$
R_j \approx \int_{x_0}^{x_j}\left( S(U) H_x + s(U,x) \right) \, dx , \quad j=0, \dots, N,$$
as follows:
	\begin{equation}\label{Rj}
	R_0 = 0, \quad {R}_{j + 1} = R_j +  I_j(U),  \quad j=0, \dots, N-1,
	\end{equation}
 where 
$$
  \mathcal{I}_j(U) \approx  \int_{x_{j}}^{x_{j + 1}} \left( S(U) H_x + s(U,x) \right)  \,dx.
$$ 

\item Define
\begin{equation}\label{mod-flux2} 
\mathcal{F}_j = F(U_j) - R_j, \quad j = i-1-r,\dots, i+s.
\end{equation}

\item  Use a high-order  flux-reconstruction operator to compute
\begin{eqnarray}
\widehat{\mathcal{F}}_{i+1/2} & =&  {\mathcal{R}}(\mathcal{F}_{i-r}, \dots, \mathcal{F}_{i+s}), 
\label{fi+1/2-2}\\
\widehat{\mathcal{F}}_{i-1/2} & = & {\mathcal{R}}(\mathcal{F}_{i-1-r}, \dots, \mathcal{F}_{i-1+s}).
\label{fi-1/2-2}
\end{eqnarray}

\end{enumerate}

A quadrature  formula that use  $U_{j + 1-s}, \dots, U_{j+1} $ to approximate the integral of the source term in the interval $[x_j, x_{j+1}]$ will be considered here
\begin{equation}\label{quad}
\begin{aligned}
  \mathcal{I}_{j}(U) := \Delta x  &\sum_{m=0}^s \beta_m S(U_{j +1 -s + m })H_x(x_{j + 1 - s +   m})\\
  +\Delta x  &\sum_{m=0}^s \beta_m 
  s(U_{j +1 -s + m },x_{j + 1 - s +   m})
  \end{aligned}
\end{equation}
whose coefficients $\beta_m$ are obtained from some multi-step ODE integrator of choice,
as e.g. those of Adams-Bashforth methods with  $s$ steps (of order $s$) if $\beta_s = 0$, or Adams-Moulton with $s$ steps  (of order $s+1$) if $\beta_s \not= 0$.
In principle, these integrators can be chosen with  accuracy greater than that of the spatial reconstruction operator.
Other methods as e.g. multi-step and spectral multi-step collocation methods can be used, see  e.g. \cite{LieNorsett,ShengEtAl}.

Observe that, if 
\begin{equation}\label{Rlinear2}
{\mathcal{R}}(\mathcal{F}_{i-r}, \dots, \mathcal{F}_{i+s})
= {\mathcal{R}}({F}_{i-r}, \dots, {F}_{i+s}) - {\mathcal{R}}\Bigl(R_{i-r}, \dots, R_{i+s} ),
\end{equation}
\eqref{wbsdmeth2} can be written in the form \eqref{sdmeth} with \eqref{numflux} and
\begin{equation}\label{source2}
S_i  = \frac{1}{\Delta x}\left({\mathcal{R}}\Bigl(R_{i-r}, \dots, R_{i+s} \Bigr)-  {\mathcal{R}}\Bigl(R_{i-1-r}, \dots, R_{i-1+s}\Bigr)  \right).
\end{equation}

Note that     \eqref{mod-flux1} and   \eqref{mod-flux2} provide a similar modification to the scheme,
the major difference being of course the way in which the subtracted flux is obtained.
Another important difference between the two methods is the following conservation property that is not satisfied by method \eqref{wbsdmeth2}:
\begin{proposition} \label{Prop:cons} Let us assume that \eqref{Rlinear2} holds. If the reconstruction procedure is computed component-by-component and it is exact for the null function,  then method \eqref{wbsdmeth2}-\eqref{quad} is conservative for the systems of conservation laws included in system \eqref{sbl}.
\end{proposition}
\begin{proof}
Assume that the $l$th equation is a conservation law, i.e. the $l$th components of $S(U)$ and $s(U,x)$ vanish. Then, from \eqref{Rj} and \eqref{quad} it is clear that the $l$th component of $R_j$ vanishes for all $j$. Finally, the assumptions on the reconstruction operator together with 
\eqref{source2} implies that the $l$th component of the numerical source term also vanishes.
\end{proof}
In particular, this result implies that, under the assumption of the Proposition, mass-conservation will be satisfied by method \eqref{wbsdmeth2} for the shallow-water equation, what is not in general the case for \eqref{wbsdmeth}: see \cite{ParesPulido}.

\begin{remark}[Evaluation of $H_x$] The derivative of $H$ can be performed analytically or
numerically with a sufficiently high accurate approximation.  For the shallow water equations,
considerations related to a particular steady state will lead to a specific approach which will be discussed in section 
\S\ref{sec_lakeatrest}.
\end{remark}

\subsection{Discrete stationary solutions} \label{ss:DSS}
Let us assume that the flux reconstruction operator is such that
\begin{equation}\label{fluxeq}
 {\mathcal{R}}(\mathcal{F}, \dots, \mathcal{F}) =  \mathcal{F}, \quad \forall \mathcal{F}.
\end{equation}
Then, a sequence of point values  $\{ U^*_j\}$ satisfying 
\begin{equation}\label{discss}
F(U^*_{j+1}) -  F(U^*_{j}) = \mathcal{I}_j(U^*) , \quad j= 0, \dots, N-1
\end{equation}
is a discrete stationary solution. 
In effect, if $\{ U^*_j \}$ satisfies \eqref{discss}, then one has:
$$
\mathcal{F}_{j+1} = F(U^*_{j+1})  - R_{j+1} =  F(U^*_{j})  - R_{j} = \mathcal{F}_{j}, \quad \forall j,
$$
so that the numerical fluxes are constant and then
$$
\widehat{\mathcal{F}}_{i,i+1/2}  - \widehat{\mathcal{F}}_{i,i-1/2} = \mathcal{F}_0 - \mathcal{F}_0 =  0,
$$
and the right-hand side of \eqref{wbsdmeth2} vanishes. Observe that \eqref{discss} is the numerical counterpart of the equality
$$
F(U^*(x_{j+1})) -  F(U^*(x_{j})) = \int_{x_j}^{x_{j+1}} \left( S(U^*) H_x + s(U^*,x) \right) \,dx, \quad j= 0, \dots, N-1
$$
satisfied by the exact stationary solution of the system.

Discrete stationary solutions satisfying \eqref{discss} can be computed as follows: given $U^*_0, \dots, U^*_{s-1}$ find $U^*_j$, $j = s, \dots, N$ such that
\begin{equation}\label{discss-multistep}
\begin{aligned}
F(U_{j}^*) = F(U_{j-1}^*) + \Delta x& \sum_{m=0}^s \beta_m S(U_{j -s m}^*)H_x(x_{j -s  + m}) \\
+ \Delta x& \sum_{m=0}^s \beta_m s(U_{j -s m}^*,x_{j -s  + m}), \quad j=s,\dots, N.
\end{aligned}
\end{equation}
Observe that a nonlinear system has to be solved to compute $U^*_{j}$ at every step and this happens even if the chosen multistep method is explicit, i.e. if
$\beta_s = 0$. In this case, $U^*_{j}$ can be computed as follows:
\begin{enumerate}
\item Compute
$$
F^*_{j} =  F(U_{j-1}^*) + \Delta x \sum_{m=0}^{s-1} \beta_m (S(U_{j - s + m}^*)H_x(x_{j -s  + m}) +s(U_{j - s + m}^*,x_{j -s  + m})).
$$

\item Find $U^*_{j}$ such that:
$$
F(U^*_{j})  =  F^*_{j}.
$$
\end{enumerate}
Observe that the second step requires the inversion of the flux function. We will assume here that the flux function
$$
U \mapsto F = F(U)$$
is one-to-one, what is true in a neighborhood of every state $U$ that is not sonic, i.e. if the eigenvalues of the Jacobian of the flux $DF(U)$ do not vanish. 

The following result holds:
\begin{theorem}[Discrete stationary states]\label{th_stationary_states}
If the flux function is a one-to-one map, then the numerical method \eqref{wbsdmeth2}-\eqref{quad} is fully well-balanced with formal  accuracy equal to the multi-step ODE solver with weights $\{\beta_m\}_{m=0,s}$.
\end{theorem}
\begin{proof}
Let $U^*$ be a stationary solution of \eqref{sbl}. Then, $F^*\equiv F(U^*)$ satisfies the Cauchy problem
$$
\left\{
\begin{array}{l}
F^*_x = {S}(U(F^*)) H_x +{s}(U(F^*),x) , \\
F^*(x_0) = F(U^*(x_0)),
\end{array}
\right.
$$
where $U(F)$ represents the inverse of the flux function.
If the selected multistep method is applied to this Cauchy problem starting from
$F^*(x_j)$, $j=0,\dots, s-1$ or from some high-order approximations of them obtained with a one-step method, a sequence of values $\{ F^*_i \}$ is obtained that satisfies for $j=s,\dots, N$
$$
F_{j}^* = F^*_{j-1} + \Delta x \sum_{m=0}^s \beta_m \big[ S(U(F_{j-s + m}^*))H_x(x_{j-s + m}) + s(U(F_{j-s + m}^*),x_{j-s + m}) \big].
$$
Then, the set of values
$$
U^*_j = U(F^*_j), \quad j=0, \dots, N$$
satisfies \eqref{discss-multistep} and, as a consequence, $\{U^*_j\}$ is a discrete stationary solution that approximates $U^*(x_j)$, $j=0, \dots, N$ with the order of accuracy of the multistep method.
\end{proof}

 \begin{remark}[Computation of $U(F)$] The  main assumption of the theorem is that $F(U)$ should be a one-to-one mapping. 
It must be noted that, strictly speaking,  for many systems this assumption may be violated in the neighborhood of sonic states. Indeed, in these points   
 the Jacobian of the flux function is singular, and
 two (or more depending on the system) branches of the function $U(F)$ meet.
 This aspect is covered for the shallow water equations  in many papers, see e.g. \cite{Kurganov2019} section \S2
 for a discussion in the context of a global flux method. 
 As shown in the reference, the inversion of $F(U)$ is well posed provided that we know on which branch
 $U$ should lay.  
 We also remark that in practice,  unlike the method \eqref{wbsdmeth}-\eqref{mod-flux1}, 
 the knowledge of  $U(F)$ is not strictly necessary here, unless the discrete stationary solutions
 are used as well-prepared initialization. The theoretical
 aspects related to sonic points are thus still unclear and   under consideration.
We will evaluate in practice the behaviour of the scheme in the results section.
Finally, the fact that no non-linear system  needs to be solved at every point, like e.g. in 
\cite{gomez2021collocation},  has some advantages in terms of  computational requirements.
 \end{remark}

Theorem \ref{th_stationary_states} shows that
the method proposed exactly preserves  discrete stationary solutions obtained with an ODE solver,
in a spirit similar to  \cite{gomez2021collocation,MANTRI2024112673}. 
As we will see in Section \ref{S:numerical}, the accuracy with which the methods approximate smooth exact stationary solutions can be  changed at will to decrease the error almost to machine accuracy. However, some particular cases deserve special attention, either because their simplicity allows an exact preservation, or because they involve inherently discontinuous solutions. This aspect is discussed in sections \S\ref{sec_lakeatrest} and \S\ref{sec_disc}.

\subsection{WENO reconstructions}\label{ss:WENO}

In the particular case of the  WENO reconstruction of order $p = 2k + 1$, two flux reconstructions are computed using the values at the points $x_{i-k}, \dots, x_{i+k}$:
\begin{eqnarray}
\widehat{\mathcal{F}}^L_{i+1/2} &  = &  {\mathcal{R}}^L(\mathcal{F}_{i-k}, \dots,\mathcal{F}_{i+k}), \\
\widehat{\mathcal{F}}^R_{i-1/2} & = &  {\mathcal{R}}^R(\mathcal{F}_{i-k}, \dots,\mathcal{F}_{i+k}). 
\end{eqnarray}
$\mathcal{R}^L$ and $\mathcal{R}^R$ represent the so-called left and right-biased reconstructions. The  right-biased reconstructions can be computed with $\mathcal{R}^L$ by reflecting the arguments through  the intercell at which the numerical flux is computed.

To compute the numerical fluxes, first a splitting 
$$
\mathcal{F}_j  = \mathcal{ F}^+_j  +\mathcal{ F}^-_j, \quad \forall j,
$$
is considered and then they are defined as follows
\begin{equation}\label{split}
\widehat{\mathcal{ F}}_{i+1/2} =  \widehat{\mathcal{ F}}^{+}_{i+1/2} + \widehat {\mathcal{F}}^{-}_{i+1/2},
\end{equation}
with
\begin{eqnarray}
& & \widehat{\mathcal{ F}}^{+}_{i+1/2}  =  {\mathcal{R}}^L(\mathcal{F}^+_{i-k}, \dots,\mathcal{F}^+_{i+k})  , \label{spllit1}\\
& &   \widehat{\mathcal{ F}}^{-}_{i+1/2}  =  {\mathcal{R}}^R(\mathcal{F}^-_{i-k+1}, \dots,\mathcal{F}^-_{i+k+1})  . \label{spllit2}
\end{eqnarray}

The general upwind biased   splitting
$$
\mathcal{F}^\pm_j = \frac{1}{2} \left( \mathcal{F}_j \pm \frac{1}{\Delta x} \mathcal{T}  J_{i+1/2} \mathcal{F}_j \right),
$$
is considered here, 
where $\mathcal{T} $ is a time-scale matrix and $J_{i+1/2}$ is an approximation of the Jacobian of the flux function $F(U)$ at the intercell (a Roe matrix, for instance). $J_{i+1/2}$ is supposed to have $N$ real different eigenvalues.  The choice
$$
\mathcal{T}  = \Delta x |J_{i+1/2}|^{-1}
$$
leads to the upwind splitting and 
$$
\mathcal{T}  = \dfrac{\Delta t}{2} .
$$
to a Lax-Wendroff type splitting. The former is used in our implementation. 
Observe that this choice guarantees \eqref{fluxeq}: if 
 $$ \mathcal{F}_j = \mathcal{F}, \quad \forall j,
$$
then
$$
 \widehat{\mathcal{ F}}^{+}_{i+1/2} = \frac{1}{2} \left( I + \frac{1}{\Delta x} \mathcal{T}  J_{i+1/2} \mathcal{F}\right) , 
 \quad \widehat {\mathcal{F}}^{-}_{i+1/2} =  \frac{1}{2} \left( I - \frac{1}{\Delta x} \mathcal{T}  J_{i+1/2} \mathcal{F} \right) ,$$
 and then
$$
\widehat{\mathcal{F}}_{i+1/2}  = \left( \frac{1}{2} \left( I + \frac{1}{\Delta x} \mathcal{T}  J_{i+1/2} \mathcal{F} \right) +  \frac{1}{2} \left( I - \frac{1}{\Delta x} \mathcal{T}  J_{i+1/2} \mathcal{F} \right) \right) \mathcal{F} = \mathcal{F}.  \label{upwinda}
$$
On the other hand, it is also clear that $\widehat{\mathcal{F}}_{i+1/2} = 0$ if $\mathcal{F}_j = 0$ for all $j$.
Moreover, since WENO reconstructions are linear combinations of the flux values, \eqref{Rlinear} or \eqref{Rlinear2} are true once the coefficients of these linear combination (that depend non-linearly of the data through the smoothness indicators) have been computed. Therefore, the hypothesis of Proposition \ref{Prop:cons} are satisfied and the well-balanced property holds.

\begin{remark} The Lax-Friedrichs flux-splitting:
$$
\mathcal{F}^\pm_j = \frac{1}{2} \left( \mathcal{F}_j \pm \alpha U \right),
$$
where $\alpha$ is the local (WENO-LLF) or global (WENO-LF)  maximum of the absolute value of the eigenvalues of $\{J(U_i)\}$, does not satisfy \eqref{fluxeq} in general. 
In effect, if
$$
\mathcal{F}_j = \mathcal{F}, \quad \forall j,
$$
one would have
$$
\widehat{\mathcal{ F}}_{i,i+1/2}  =  \mathcal{F} - \frac{\alpha}{2} \left(    {\mathcal{R}}^L(U_{i-k}, \dots,U_{i+k}) -  {\mathcal{R}}^R(U_{i-k+1}, \dots,U_{i+k+1})  \right)
\not= \mathcal{F}.
$$
\end{remark}


\section{Enhancements for singular sources and  water at rest}

\subsection{Discontinuous data and singular source terms}\label{sec_disc}

Let us assume  now  the case in which $s=0$, and  the function $H$ has at least one isolated  discontinuity. In this case, the definition of weak solutions (and, in particular, of stationary solutions)  of \eqref{sbl}  becomes more difficult:  a solution $U$ is expected to be discontinuous at the  discontinuities of $H$ and, in this case,  the source term
$ S(U)H_x$ cannot be defined within the distributional framework. The source term becomes then a non-conservative product that can be defined in infinitely many different forms: see \cite{DalMaso95}.
Like in \cite{Castro-CP2020}, the following criterion is assumed here to define the weak solutions of \eqref{sbl}:  a pair $([U^-, H^-], [U^+, H^+]))$  can be the  limits of $U$ and $H$ of an admissible weak solution of \eqref{sbl} to the left and to the right of  a discontinuity point of  $H$, $x^*$,   if and only if there exists a solution of the 
 ODE system
\begin{equation}\label{jumpcondition1}
 \frac{d \ }{d H}F(U) = S(U)
\end{equation}
satisfying
\begin{equation}\label{jumpcondition2}
U(H^\pm) =  U^\pm.
\end{equation}

This criterion implies for instance that, in the case of Burgers' equation with source term, the relations \eqref{ssBurgers2} or \eqref{ssBurgersp} have to be satisfied by admissible jumps at the discontinuities of $H$. 
In the case of the shallow water system, relations \eqref{ssSW} must be satisfied at the discontinuities of the bottom function.


Let us suppose for simplicity that $H$ has a unique discontinuity point that is placed at an interface  $x_{\ell+1/2}$.
In order to avoid the appearance of oscillations in the computation of discrete stationary solutions with the selected multistep method, an adaptive technique is used in which the number of steps is adjusted so that $x_{\ell+1/2}$ always stays out of the stencil, except when $U^*_{\ell+1}$ is computed. The notation
$$\beta^r_m, \quad m= 0,\dots, r, \quad r = 2, \dots, s,$$
will be used in what follows to represent the weights of the $r$-step method. Using this notation, discrete stationary solutions are computed now as follows: given $U^*_0, \dots, U^*_{s-1}$ 
\begin{itemize}
\item For  $j=s,\dots, \ell$ compute $U^*_{j}$ by 
\begin{equation}\label{firstlaststeps}
F(U_{j}^*) = F(U_{j-1}^*) + \Delta x \sum_{m=0}^s \beta^s_m S(U_{j - s + m}^*)H_x(x_{j -s  + m}), 
\end{equation}

\item For $j = \ell + 1$ compute $U^*_{\ell+1}$ so that $([U^*_{\ell}, H^-(x_{\ell})], [U^*_{\ell+1}, H^+(x_{\ell+1})])$ is an admissible jump   $x_{\ell+1/2}$.

\item For $j = \ell + 2, \dots, \ell + s$ compute $U^*_{j}$ by 
\begin{equation}\label{midsteps}
F(U_{j}^*) = F(U_{j-1}^*) + \Delta x \sum_{m=0}^r \beta^r_m S(U_{j - r + m}^*)H_x(x_{j -r  + m}), 
\end{equation}
with $r = j - l - 1$.

\item For  $j = \ell + s + 1, \dots, N$ compute $U^*_{j}$ by \eqref{firstlaststeps}.

\end{itemize}

Once the computation of the discrete stationary solutions has been set, let us introduce the procedure to approximate 
$$
\int_{x_j}^{x_{j+1}} S(U) H_x(x) dx
$$
so that the numerical method preserves them:

\begin{itemize}
\item if $j < \ell $, proceed as usual 
\begin{equation}\label{firstlastints}
  \mathcal{I}_{j}(U) = \Delta x \sum_{m=0}^s \beta^s_m S(U_{j +1 -s + m })H_x(x_{j + 1 - s +   m});
\end{equation}
\item if $j = \ell$ we set
\begin{equation}\label{discint}
 \mathcal{I}_{\ell}(U) = \tilde S_{\ell +1/2} [\![ H]\!]_{\ell+1/2}
\end{equation}
where $\tilde S_{i+1/2}$ is a linearization of $S(U)$ such that, across an admissible jump at the discontinuity point of $H$, one has
\begin{equation} \label{linearS(U)}
[\![ F]\!]_{i+1/2}= \tilde S_{i+1/2}[\![ H]\!]_{i+1/2};
\end{equation}
\item  if $j = \ell+ r$ with $1 \leq r \leq s$, 
\begin{equation}\label{midints}
  \mathcal{I}_{j}(U) = \Delta x \sum_{m=0}^r \beta^r_m S(U_{\ell + 1 + m })H_x(x_{\ell + 1 +   m});
\end{equation}
\item  $j > \ell + s $, proceed as usual \eqref{firstlastints}.
\end{itemize}


Exact formulas for the linearization of $S(U)$ for the prototype equations are provided in appendix \ref{app_linearization}.\\

An alternative to \eqref{discint} for computing  $\mathcal{I}_{\ell}(U) $ so that \eqref{discss} is satisfied for $j = \ell$ is given by the so-called Generalized Hydrostatic Reconstruction (GHR): see \cite{GHR2007}. To compute the integral in this case, first an intermediate value $H_{\ell + 1/2}$ has to be chosen between $H(x_{\ell})$ and $H(x_{\ell + 1})$; next two states 
$U_{\ell + 1/2}^\pm$ have to be computed such that $\left([U_{\ell}, H(x_{\ell})],  [U^-_{\ell+1/2}, H_{\ell + 1/2}]\right)$ and $\left( [U^+_{\ell+1/2}, H_{\ell + 1/2}], [U_{\ell+1}, H(x_{\ell+1})]  \right)$ are admissible jumps. 
Then,
$$
\mathcal{I}_{\ell}(U)  =  F(U_{\ell+1}) - F(U^+_{\ell+1/2})+  F(U^-_{\ell+1/2}) - F(U_{\ell}).$$
When $\left([U_{\ell}, H(x_{\ell})], [U_{\ell+1}, H(x_{\ell+1})]  \right)$ is an admissible jump then
$U^-_{\ell+1/2} = U^+_{\ell+1/2}$ and \eqref{discss} is satisfied. In practice, $H_{\ell + 1/2}$ can be taken equal to
$H(x_{\ell})$ or to $H(x_{\ell + 1})$ so that only one intermediate state has to be computed: see \cite{GHR2007} for details.

\subsection{Exact preservation of water-at-rest solutions}\label{sec_lakeatrest}

As discussed at the end of  section \S\ref{ss:DSS}, the proposed method preserves exactly some discrete stationary solutions but not the exact ones.  In particular, 
in the case of the shallow water equations, the numerical methods as described above will not preserve exactly the point values of water-at-rest stationary solutions.
To be more precise, they will yield some discrete approximations which  satisfy exactly $q = 0$, but verify  $\eta = constant$ only up to the order of accuracy of the selected multistep method. 
In this sense, the $C$-property stated in \cite{Bermudez-Vazquez1994} is not satisfied. Nevertheless, it is easy to modify the method to recover this property.

To do this, we first compute approximations $\tilde H_i \approx H(x_i)$ such that
$$
\widetilde H_{i} = \widetilde H_{i-1} + \Delta x \sum_{m=0}^s \beta_m H_x(x_{j -s  + m}), 
$$
i.e. the bottom function is approximated by applying the selected quadrature formula to the expression
$$
H(x_i) = H(x_{i-1}) + \int_{x_{i-1}}^{x_i} H_x(x) \, dx.$$
We now consider the   equality 
 \begin{equation}\label{intwar}
 \int_{x_i}^{x_{i+1}} g h H_x \,dx =  \int_{x_i}^{x_{i+1}} g \eta H_x \,dx + \frac{g}{2} [\![H^2]\!]_{i+1/2} . \end{equation}
Based on this equality and introducing the approximations of $H$,  the second component of $\mathcal{I}_{j}(U)$ is computed as follows
 \begin{equation}\label{quadSW1}
  \mathcal{I}_{2,j}(U) = \Delta x \sum_{m=0}^s \beta_m g \eta_{j +1 -s + m } H_x(x_{j + 1 - s +   m}) + \frac{g}{2} [\![(\widetilde{H})^2]\!]_{i+1/2}.
\end{equation}
Let us see that the sequence of point values of a water-at-rest stationary solution $U^*_j = [ h^*_j, q^*_j]^T$, where
$$
q^*_j = 0,\quad h^*_j =  \eta^* + \widetilde{H}_j, \quad \eta^*  = constant,$$
is a discrete stationary solution, i.e. that \eqref{discss} is satisfied. This equality is trivially satisfied for the first components. Let us check the second one: taking into account that the quadrature formula is exact for $H^M_x$, one has
\begin{eqnarray*}
  \mathcal{I}_{2,j}(U)  & = & \Delta x \sum_{m=0}^s \beta_m g \eta^*_{j +1 -s + m } H_x(x_{j + 1 - s +   m}) + \frac{g}{2} [\![(\widetilde{H})^2]\!]_{i+1/2}\\
  & = &   g \eta^*  [\![\widetilde{H}]\!]_{i+1/2}   + \frac{g}{2} [\![(\widetilde{H})^2]\!]_{i+1/2} \\
  & = & g\eta^*  [\![\widetilde{H}]\!]_{i+1/2}   +  {g}  \overline{\widetilde{H}}_{i+1/2}[\![\widetilde{H}]\!]_{i+1/2} \\
    & = & g\overline{\eta^*}_{i+1/2}  [\![h]\!]_{i+1/2}   + g  \overline{\widetilde{H}}_{i + 1/2} [\![h]\!]_{i+1/2} \\
       & = & g\overline {h^*}_{i+1/2} [\![h]\!]_{i+1/2} \\
       & = &  [\![g(h^*)^2/2]\!]_{i+1/2},
       \end{eqnarray*}
where $\overline{w}_{i+1/2}$ represents the arithmetic mean of $w_i, w_{i+1}$ for any variable $w$. Therefore, \eqref{discss} is satisfied as we wanted to prove.

Another possibility is to  proceed as in \cite{MANTRI2024112673}  and  consider a polynomial interpolation of    $H$ of degree $M = s$, denoted by   $H^M$. Then, the second component of $\mathcal{I}_{j}(U)$ is computed as follows
 \begin{equation}\label{quadSW2}
  \mathcal{I}_{2,j}(U) = \Delta x \sum_{m=0}^s \beta_m g \eta_{j +1 -s + m } H^M_x(x_{j + 1 - s +   m}) + \frac{g}{2} [\![(H^M)^2]\!]_{i+1/2}.
\end{equation}
In this case, the sequence of point values of a water-at-rest stationary solution $U^*_j = [ h^*_j, q^*_j]^T$, where
$$
q^*_j = 0,\quad h^*_j =  \eta^* + H(x_j), \quad \eta^*  = constant,$$
is a discrete stationary solution, i.e. that \eqref{discss} is satisfied: the proof is similar, taking into account that the quadrature formula is exact for $H^M_x$.

Both approaches have different appeals. The first provides in practice lower errors fir smooth bathymetries.
The second  manages better bathymetries with lower regularity. Since for discontinuous bathymetries we use the local modification
discussed in the previous section,   we have chosen to use the first  method in the numerical applications.

\section{Implementation details}
\label{section:implementation_details}
In Section \ref{S:numerical} some numerical tests will be shown in which the implementation of the methods introduced in Section \ref{S:methods} is as follows: 
\begin{itemize}
\item WENO reconstruction operators based on the upwind splitting are used (see \ref{ss:WENO}). Like in \cite{ParesPulido},  Jiang-Shu smoothness indicators (see \cite{JiangShu})  are used here.
\item The coefficients of multi-step Adams-Bashforth method of order $q$ (AB$q$) or Adams-Moulton of order $q$ (AB$q$) will be used to compute the integrals of the source term. These coefficients up to order 8 are shown in appendix \ref{app_AM-AMcoeff}.
\item TVD-RK time stepping is applied to the semi-discrete in space numerical methods
\eqref{wbsdmeth2} (see \cite{Gottlieb98}).
\item CFL statility condition 
$$\Delta t = CFL \frac{\Delta x}{( \max_i \{ \lvert \lambda_{j,i+1/2} \rvert, 1 \leq l \leq N\}) } $$
is imposed with $CFL \in (0,1)$, where $\lambda_{j, i+1/2}$ are the eigenvalues of the Jacobian of the flux function computed at some intermediate states. Unless otherwise stated $CFL = 0.45$.
\item Boundary conditions are imposed through the use of ghost nodes.
\end{itemize}

Taking into account these ingredients, two families of methods are considered:
\begin{itemize}
\item WENO$p$GF-AB$q$: WENO reconstructions of order $p=3,5, 7$ with the Global Flux approach using Adams-Bashforth method of order  $q=4,6,8$ ($s = q$ steps) to compute the integrals;

\item WENO$p$GF-AM$q$: WENO reconstructions of order $p=3,5,7$  with the Global Flux approach using Adams-Moulton method  of order $q=4,6,8$ ($s = q-1$ steps);
\end{itemize}
In some test cases, the numerical results will be compared with non-well balanced standard WENO implememtations that will be called WENO$p$-nWB. In cetrain cases, the discrete steady-state solution is computed using a sweep of the  Adams-Bashforth multi-step solver. This  involves finding the roots of a  degree three  equation which can be done very easily if the super/sub- critical nature of the flow is known (see  e.g.  \cite{chen2017new}).  A similar implementation can also been performed for the  Adams-Moulton method, which however requires  fixed point iterations instead of the direct  in the computation of the cubic roots.

Since in many cases the numerical results given by WENO$p$GF-AB$q$ and WENO$p$GF-AM$q$ are similar, only those corresponding to the latter will be shown, unless otherwise stated. 

\section{Numerical validation}\label{S:numerical}

\subsection{Burgers' equation}

We solve in this section  the scalar Burgers'  equation
\begin{equation}\label{testburgers}
U_t +  \left(\frac{U^2}{2}\right) _x = S(U)H_x,
\end{equation}
We will consider time dependent, steady smooth, and discontinuous solutions.

\subsubsection{Manufactured traveling solution}\label{ss:num_scalar_moving}

As a first safety check, we consider the manufactured solution
$$
 U(x,t)=e^{-(x-x_0-Ct)^2}
 $$
 admitted when   setting   $$ S(U)=U-C, \quad  H(x,t)=e^{-(x-x_0-Ct)^2}$$
 in the Burgers' model. We solve the problem with $C=1$ and $x_0=5$  on the computational domain $[0,1 5]$ m. We compute the error of the solution in $L^1$ norm at time $t=2$ sec for the different WENO reconstruction and multi-step (Adams) methods. Tables 1, 2 and 3 show the errors and measured order of convergence of the different methods.  As one might expect for a time dependent solution, the order of accuracy is given by the minimum between the one of the WENO reconstruction, and that of the ODE integrator. 

\begin{table}[h]
\centering
\caption{Burger's equation. Test \ref{ss:num_scalar_moving}: traveling manufactured solution. $L^1$ error convergence for WENO$p$GF-AM4 and  WENO$p$GF-AB4, $p=3,5,7$}
\label{Table_WENO_4steps_mms}
\begin{tabular}{ |c|c@{\hskip 5pt}c|c@{\hskip 5pt}c|c@{\hskip 5pt}c| }
\hline
 & \multicolumn{2}{|c|}{\textbf{WENO3GF - AB4}} & \multicolumn{2}{|c|}{\textbf{WENO5GF - AB4}} & \multicolumn{2}{|c|}{\textbf{WENO7GF - AB4}} \\
 \hline
 N & Error & Order & Error & Order & Error & Order \\
 \hline
 60  & 1.393e-01 & -    & 1.451e-01 & -    & 1.160e-01 & -   \\
 120 & 1.515e-02 & 3.2  & 1.250e-02 & 3.5  & 1.223e-02 & 3.7 \\
 240 & 1.293e-03 & 3.5  & 8.264e-04 & 3.9  & 8.166e-04 & 3.9 \\
 480 & 1.212e-04 & 3.4  & 5.200e-05 & 4.0  & 5.173e-05 & 4.0 \\
 960 & 1.306e-05 & 3.2  & 3.259e-06 & 4.0  & 3.251e-06 & 4.0 \\
\hline
 & \multicolumn{2}{|c|}{\textbf{WENO3GF - AM4}} & \multicolumn{2}{|c|}{\textbf{WENO5GF - AM4}} & \multicolumn{2}{|c|}{\textbf{WENO7GF - AM4}} \\
\hline
N & Error & Order & Error & Order & Error & Order \\
\hline
60  & 3.661e-02 & -   & 3.659e-02 & -   & 1.273e-01 & -   \\
120 & 5.220e-03 & 2.8 & 3.103e-03 & 3.5 & 9.659e-04 & 3.9 \\
240 & 7.020e-04 & 2.9 & 2.531e-04 & 3.6 & 6.254e-05 & 4.0 \\
480 & 9.016e-05 & 2.9 & 1.816e-05 & 3.8 & 2.465e-07 & 4.0 \\
960 & 1.140e-06 & 2.9 & 1.195e-06 & 3.9 & 1.541e-08 & 4.0 \\
\hline
\end{tabular}
\end{table}

\begin{table}[h]
\centering
\caption{Burger's equation. Test \ref{ss:num_scalar_moving}: traveling manufactured  solution. $L^1$ error convergence for WENO$p$GF-AM6 and  WENO$p$GF-AB6, $p=3,5,7$}
\label{Table_WENO_6steps_mms}
\begin{tabular}{ |c|c@{\hskip 5pt}c|c@{\hskip 5pt}c|c@{\hskip 5pt}c| }
\hline
 & \multicolumn{2}{|c|}{\textbf{WENO3GF - AB6}} & \multicolumn{2}{|c|}{\textbf{WENO5GF - AB6}} & \multicolumn{2}{|c|}{\textbf{WENO7GF - AB6}} \\
 \hline
 N & Error & Order & Error & Order & Error & Order \\
 \hline
 60  & 1.377e-01 & -    & 1.204e-01 & -    & 1.552e-01 & - \\
 120 & 5.331e-03 & 4.6  & 4.081e-03 & 4.8  & 4.306e-03 & 5.1 \\
 240 & 6.934e-04 & 2.9  & 6.571e-05 & 6.0  & 7.551e-05 & 5.8 \\
 480 & 9.127e-05 & 2.9  & 9.642e-07 & 6.0  & 1.217e-06 & 5.9 \\
 960 & 1.149e-05 & 2.9  & 1.905e-08 & 5.6  & 1.905e-08 & 6.0 \\
\hline
 & \multicolumn{2}{|c|}{\textbf{WENO3GF - AM6}} & \multicolumn{2}{|c|}{\textbf{WENO5GF - AM6}} & \multicolumn{2}{|c|}{\textbf{WENO7GF - AM6}} \\
\hline
N & Error & Order & Error & Order & Error & Order \\
\hline
60  & 4.419e-02  & -    &  1.524e-02  & -   & 1.090e-02  & - \\ 
120 & 5.840e-03  & 2.9  &  7.391e-04  & 4.3 & 2.156e-04  & 5.9 \\
240 & 7.320e-04  & 2.9  &  2.278e-05  & 5.0 & 3.622e-06  & 5.9 \\
480 & 9.193e-05  & 2.9  &  6.921e-07  & 5.0 & 5.955e-08  & 5.7 \\
960 & 1.150e-05  & 2.9  &  2.159e-08  & 5.0  & 7.418e-11 & 4.0 \\
\hline
\end{tabular}
\end{table}

\begin{table}[h]
\centering
\caption{Burger's equation. Test \ref{ss:num_scalar_moving}: traveling manufactured  solution. $L^1$ error convergence for WENO$p$GF-AM8 and  WENO$p$GF-AB8, $p=3,5,7$}
\label{Table_WENO_8steps_mms}
\begin{tabular}{ |c|c@{\hskip 5pt}c|c@{\hskip 5pt}c|c@{\hskip 5pt}c| }
\hline
 & \multicolumn{2}{|c|}{\textbf{WENO3GF - AB8}} & \multicolumn{2}{|c|}{\textbf{WENO5GF - AB8}} & \multicolumn{2}{|c|}{\textbf{WENO7GF - AB8}} \\
 \hline
 N & Error & Order & Error & Order & Error & Order \\
 \hline
 60  & 1.2846751 & -    & 4.0572682 & -    & 5.2918144 & -   \\
 120 & 5.750e-03 & 7.8  & 1.8716079 & 1.1  & 1.2131538 & 2.1 \\
 240 & 7.331e-04 & 2.9  & 1.638e-02 & 6.8  & 4.032e-02 & 4.9 \\
 480 & 9.192e-05 & 2.9  & 9.845e-05 & 7.3  & 4.848e-06 & 13.0 \\
 960 & 1.150e-05 & 2.9  & 5.762e-07 & 7.7  & 5.588e-07 & 3.1 \\
\hline
 & \multicolumn{2}{|c|}{\textbf{WENO3GF - AM8}} & \multicolumn{2}{|c|}{\textbf{WENO5GF - AM8}} & \multicolumn{2}{|c|}{\textbf{WENO7GF - AM8}} \\
\hline
N & Error & Order & Error & Order & Error & Order \\
\hline
 60  &   4.419e-02  & -    &  1.524e-02  & -    & 1.090e-02  & -   \\
 120 &   5.840e-03  & 2.9  &  7.391e-04  & 4.3  & 2.156e-04  & 5.9 \\
 240 &   7.320e-04  & 2.9  &  2.278e-05  & 5.0  & 3.622e-06  & 5.9 \\
 480 &   9.193e-05  & 2.9  &  6.921e-07  & 5.0  & 5.955e-08  & 5.7 \\
 960 &   1.150e-05  & 2.9  &  2.159e-08  & 5.0  & 7.418e-11  & 4.0 \\
\hline
\end{tabular}
\end{table}

\begin{table}[h]
\centering
\caption{Burgers' equation. Test \ref{ss:num_scalar_moving}: traveling manufactured  solution. $L^1$ error and convergence rates for WENO$p$-nWB, $p=3,5,7$}
\label{Table_WENOnWB_mms}
\begin{tabular}{|c|c c|c c|c c|}
\hline
 & \multicolumn{2}{c|}{\textbf{WENO3-nWB}} & \multicolumn{2}{c|}{\textbf{WENO5-nWB}} & \multicolumn{2}{c|}{\textbf{WENO7-nWB}} \\
\hline
 N  & Error  & Order & Error  & Order & Error  & Order \\
\hline
 60  & 5.734e-02 & -    & 2.335e-02 & -    & 8.761e-03 & -   \\
 120 & 8.648e-03 & 2.7  & 1.634e-03 & 3.8  & 1.403e-04 & 6.0 \\
 240 & 1.140e-03 & 2.9  & 6.376e-05 & 4.7  & 1.678e-06 & 6.3 \\
 480 & 1.447e-04 & 2.9  & 2.062e-06 & 4.9  & 2.232e-08 & 6.2 \\
 960 & 1.815e-05 & 3.0  & 6.546e-08 & 5.0  & 3.262e-08 & 6.0 \\
\hline
\end{tabular}
\end{table}

\subsubsection{Approximation of a smooth steady state}\label{ss:num_scalar_steady}
We now consider the case 
$$
S(U) = U^2, \quad H(x) = x, $$
 together with its stationary solution
\begin{equation}
U^*(x) = e^x
\end{equation}
in the interval  $[-1, 1]$. For this case, we expect the ODE-based GF formulation to bring considerable accuracy enhancements.
To check this we use the exact point values of the stationary solution as initial data, and then   run until a discrete stationary solution is reached. The error in $L^1$-norm of the discrete steady state  with respect to the analytical one is then measured. Tables \ref{Table_WENO_4steps_steady}-\ref{Table_WENO_8steps_steady} show the errors and empirical order of convergence of the different methods. It can be seen that the order of accuracy as well as the error values for the GF case are uniquely determined by the choice of the multi-step method. This allows for  error reductions of several orders of magnitude on a given mesh, just by modifying the weights of the quadrature formula.

\begin{table}[h!]
\centering
\caption{Burger's equation. Test \ref{ss:num_scalar_steady}:  smooth steady state. Errors in $L^1$ norm and convergence rate for  WENO$p$GF-AM4 and WENO$p$GF-AB4, $p=3,5,7$}
\label{Table_WENO_4steps_steady}
\begin{tabular}{|p{0.5cm}|c@{\hskip 5pt}c|c@{\hskip 5pt}c|c@{\hskip 5pt}c|}
\hline
 & \multicolumn{2}{|c|}{\textbf{WENO3GF - AB4}} & \multicolumn{2}{|c|}{\textbf{WENO5GF - AB4}} & \multicolumn{2}{|c|}{\textbf{WENO7GF - AB4}} \\
\hline
N & Error & Order & Error & Order & Error & Order \\
\hline
20 & 1.380e-03 & - & 1.308e-03 & - & 1.380e-03 & - \\
40 & 9.628e-05 & 3.8 & 9.388e-05 & 3.8 & 9.628e-05 & 3.8 \\
80 & 6.359e-06 & 3.9 & 6.282e-06 & 3.9 & 6.359e-06 & 3.9 \\
160 & 4.086e-07 & 4.0 & 4.086e-07 & 4.0 & 4.086e-07 & 4.0 \\
320 & 2.590e-08 & 4.0 & 2.589e-08 & 4.0 & 2.900e-08 & 4.0 \\
\hline
 & \multicolumn{2}{|c|}{\textbf{WENO3GF - AM4}} & \multicolumn{2}{|c|}{\textbf{WENO5GF - AM4}} & \multicolumn{2}{|c|}{\textbf{WENO7GF - AM4}} \\
\hline
N & Error & Order & Error & Order & Error & Order \\
\hline
20 & 1.112e-04 & - & 1.115e-04 & - & 1.115e-04 & - \\
40 & 7.865e-06 & 3.9 & 7.670e-06 & 3.9 & 7.669e-06 & 3.9 \\
80 & 5.000e-07 & 3.9 & 5.000e-07 & 3.9 & 4.939e-07 & 3.9 \\
160 & 3.152e-08 & 4.0 & 3.152e-08 & 4.0 & 3.155e-08 & 4.0 \\
320 & 1.979e-09 & 4.0 & 1.979e-09 & 4.0 & 1.978e-09 & 4.0 \\
\hline
\end{tabular}
\end{table}

\begin{table}[h!]
\centering
\caption{Burgers' equation. Test \ref{ss:num_scalar_steady}: smooth steady state. Errors in $L^1$ norm and convergence rate for WENO$p$GF-AM6 and WENO$p$GF-AB6, $p=3,5,7$ }
\label{Table_WENO_6steps_steady}
\begin{tabular}{|p{0.5cm}|c@{\hskip 5pt}c|c@{\hskip 5pt}c|c@{\hskip 5pt}c|}
\hline
 & \multicolumn{2}{|c|}{\textbf{WENO3GF - AB6}} & \multicolumn{2}{|c|}{\textbf{WENO5GF - AB6}} & \multicolumn{2}{|c|}{\textbf{WENO7GF - AB6}} \\
\hline
N & Error & Order & Error & Order & Error & Order \\
\hline
20 & 3.917e-05 & - & 3.898e-05 & - & 3.889e-05 & - \\
40 & 7.909e-07 & 5.7 & 7.862e-07 & 5.7 & 7.704e-07 & 5.7 \\
80 & 1.370e-08 & 5.9 & 1.369e-08 & 5.9 & 1.371e-08 & 5.8 \\
160 & 2.256e-10 & 5.9 & 2.256e-10 & 5.9 & 2.255e-10 & 5.9 \\
320 & 3.737e-12 & 5.9 & 3.774e-12 & 5.9 & 3.762e-12 & 5.9 \\
\hline
 & \multicolumn{2}{|c|}{\textbf{WENO3GF - AM6}} & \multicolumn{2}{|c|}{\textbf{WENO5GF - AM6}} & \multicolumn{2}{|c|}{\textbf{WENO7GF - AM6}} \\
\hline
N & Error & Order & Error & Order & Error & Order \\
\hline
20 & 2.202e-06 & - & 2.124e-06 & - & 2.088e-06 & - \\
40 & 3.884e-08 & 5.8 & 3.879e-08 & 5.8 & 3.890e-08 & 5.8 \\
80 & 6.459e-10 & 5.9 & 6.455e-10 & 5.9 & 6.453e-10 & 5.9 \\
160 & 1.035e-11 & 6.0 & 1.035e-11 & 6.0 & 1.035e-11 & 6.0 \\
320 & 4.159e-16 & - & 2.331e-15 & - & 1.993e-15 & - \\
\hline
\end{tabular}
\end{table}

\begin{table}[h!]
\centering
\caption{Burgers' equation. Test \ref{ss:num_scalar_steady}: smooth steady state. Errors in $L^1$ norm and convergence rate for WENO$p$GF-AM8 and WENO$p$GF-AB8, $p=3,5,7$}
\label{Table_WENO_8steps_steady}
\begin{tabular}{|p{0.5cm}|c@{\hskip 5pt}c|c@{\hskip 5pt}c|c@{\hskip 5pt}c|}
\hline
 & \multicolumn{2}{|c|}{\textbf{WENO3GF - AB8}} & \multicolumn{2}{|c|}{\textbf{WENO5GF - AB8}} & \multicolumn{2}{|c|}{\textbf{WENO7GF - AB8}} \\
\hline
N & Error & Order & Error & Order & Error & Order \\
\hline
20 & 1.267e-06 & - & 2.071e-05 & - & 1.694e-05 & - \\
40 & 6.702e-09 & 7.6 & 5.353e-04 & - & 5.875e-05 & - \\
80 & 3.056e-11 & 7.8 & 1.338e-06 & 8.6 & 4.962e-06 & 3.6 \\
160 & 8.677e-14 & 8.5 & 1.160e-13 & - & 3.425e-08 & 7.2 \\
320 & - & - & - & - & 6.657e-14 & - \\
\hline
 & \multicolumn{2}{|c|}{\textbf{WENO3GF - AM8}} & \multicolumn{2}{|c|}{\textbf{WENO5GF - AM8}} & \multicolumn{2}{|c|}{\textbf{WENO7GF - AM8}} \\
\hline
N & Error & Order & Error & Order & Error & Order \\
\hline
20 & 4.792e-08 & - & 4.765e-08 & - & 4.718e-08 & - \\
40 & 2.317e-10 & 7.7 & 2.314e-10 & 7.7 & 2.312e-10 & 7.7 \\
80 & 9.666e-13 & 7.9 & 9.688e-13 & 7.9 & 9.828e-13 & 7.9 \\
160 & 1.346e-16 & - & 1.429e-16 & - & 1.540e-16 & - \\
320 & - & - & - & - & - & - \\
\hline
\end{tabular}
\end{table}

\begin{table}[h!]
\centering
\caption{Burgers' equation. Test \ref{ss:num_scalar_steady}: smooth steady state. Errors in $L^1$ norm and convergence rates for WENO$p$-nWB, with $p=3,5,7$}
\label{Table_WENOnWB_staedy}
\begin{tabular}{|c|c c|c c|c c|}
\hline
 & \multicolumn{2}{c|}{\textbf{WENO3-nWB}} & \multicolumn{2}{c|}{\textbf{WENO5-nWB}} & \multicolumn{2}{c|}{\textbf{WENO7-nWB}} \\
\hline
 N  & Error  & Order & Error  & Order & Error  & Order \\
\hline
20  & 9.531e-03 & -   & 9.387e-05 & -   & 5.073e-07 & -   \\
40  & 1.166e-03 & 3.0 & 3.468e-06 & 4.8 & 5.633e-09 & 6.5 \\
80  & 1.431e-04 & 3.0 & 1.112e-07 & 5.0 & 4.571e-11 & 6.9 \\
160 & 1.767e-05 & 3.0 & 3.503e-09 & 5.0 & 1.253e-13 & 8.5 \\
320 & 2.193e-06 & 3.0 & 1.094e-10 & 5.0 & -         & -   \\
\hline
\end{tabular}
\label{TableWENOComparison}
\end{table}

\subsubsection{Highly oscillatory solution}\label{ss:num_oscil}
As in \cite{ParesPulido} we consider now \eqref{testburgers} with an oscillatory  function $H$:
\begin{equation} \label{H_osc}
S(U) = U^2, \quad H(x) = x + 0.1 \sin(100x),
\end{equation}
together with the stationary solution
$$
U^*(x) = e^{H(x)}$$
in the interval $[-1,1]$
(see Figure \ref{fig-stBHo}). We take as initial condition the stationary solution and first a 100-cell mesh is considered, so that the period of the oscillations is close to $\Delta x$.
In Figure \ref{Fig_test_oscill} (top-left) the numerical solutions obtained at time $t = 1 $ sec with WENO3GF-AM$q$, $q = 4,6,8$ and WENO$p$-nWB, $p = 3, 5$ are compared with the exact stationary solution. As expected the well-balanced methods behave better than the non-WB ones even when their order of accuracy is lower. Nevertheless, the choice $q = 8$ is surprisingly worse than $q=4$ and $q=6$. This is a consequence of the large stencil introduced
by the eighth order method, which is not suited to approximate
the highly oscillatory solution considered.
To verify this   we consider a slight refinement of the mesh.
In top right picture of Figure \ref{Fig_test_oscill} 
we thus report the results obtained using 150 points.
It can be seen that all methods give improved results. In this case   WENO3GF-AM8 is   the best one as shown by the  zoom in the bottom picture  in Figure \ref{Fig_test_oscill}.\\

\begin{figure}
\centering
\includegraphics[width=0.485\textwidth]{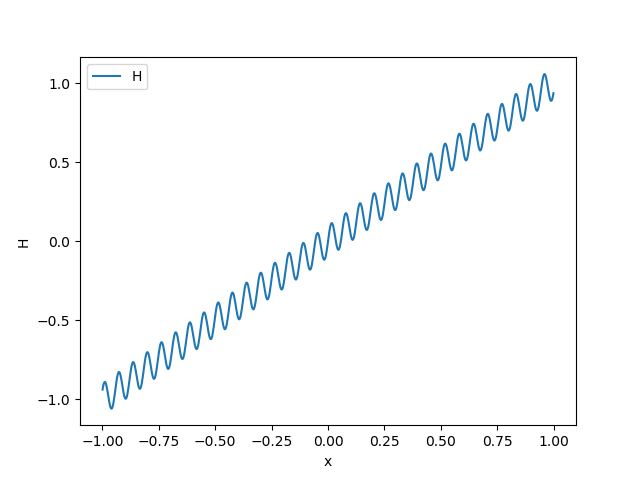}
\caption{Burgers' equation. Test \ref{ss:num_oscil}: highly oscillatory solution. Graph of the function $H$ \label{fig-stBHo}}
\end{figure}



\begin{figure}
\includegraphics[width=0.55\textwidth]{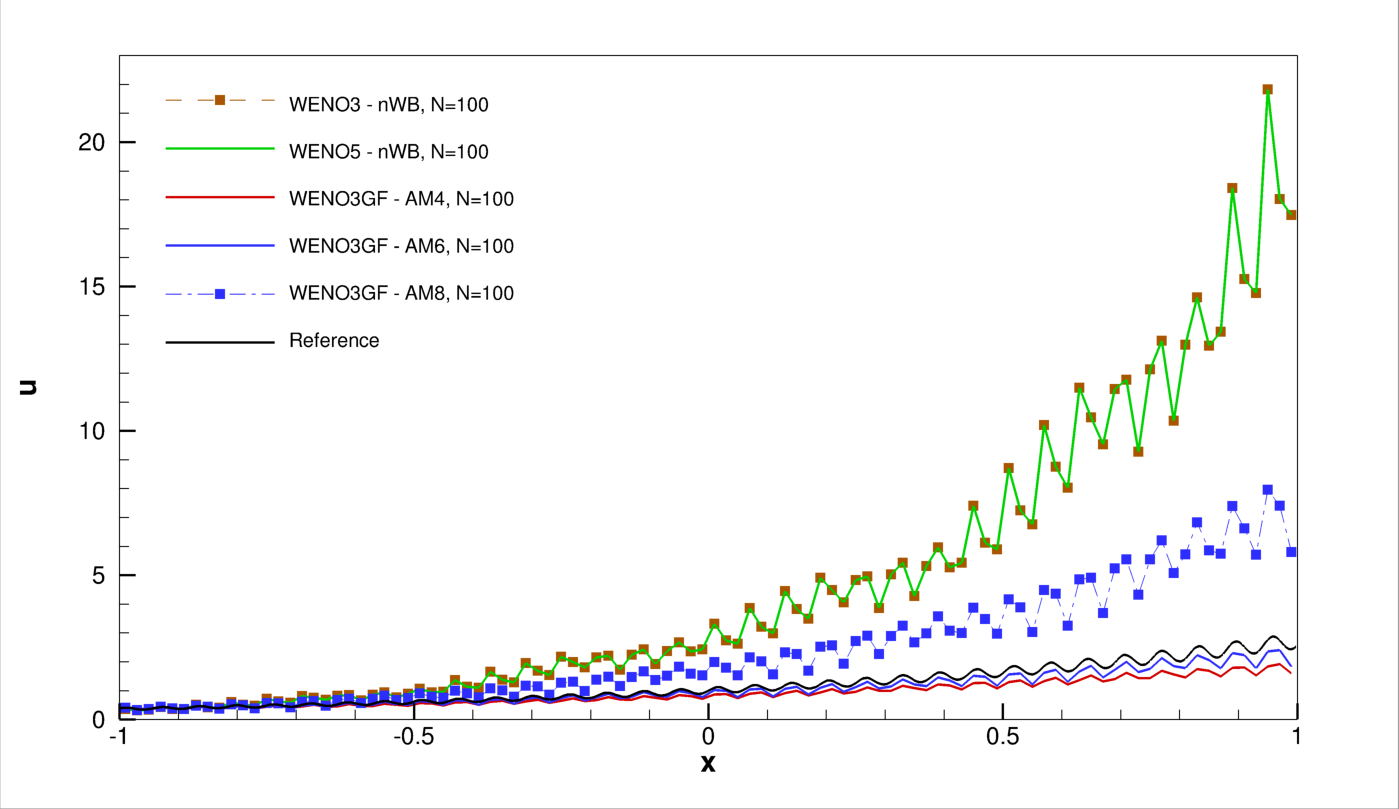}
\includegraphics[width=0.55\textwidth]{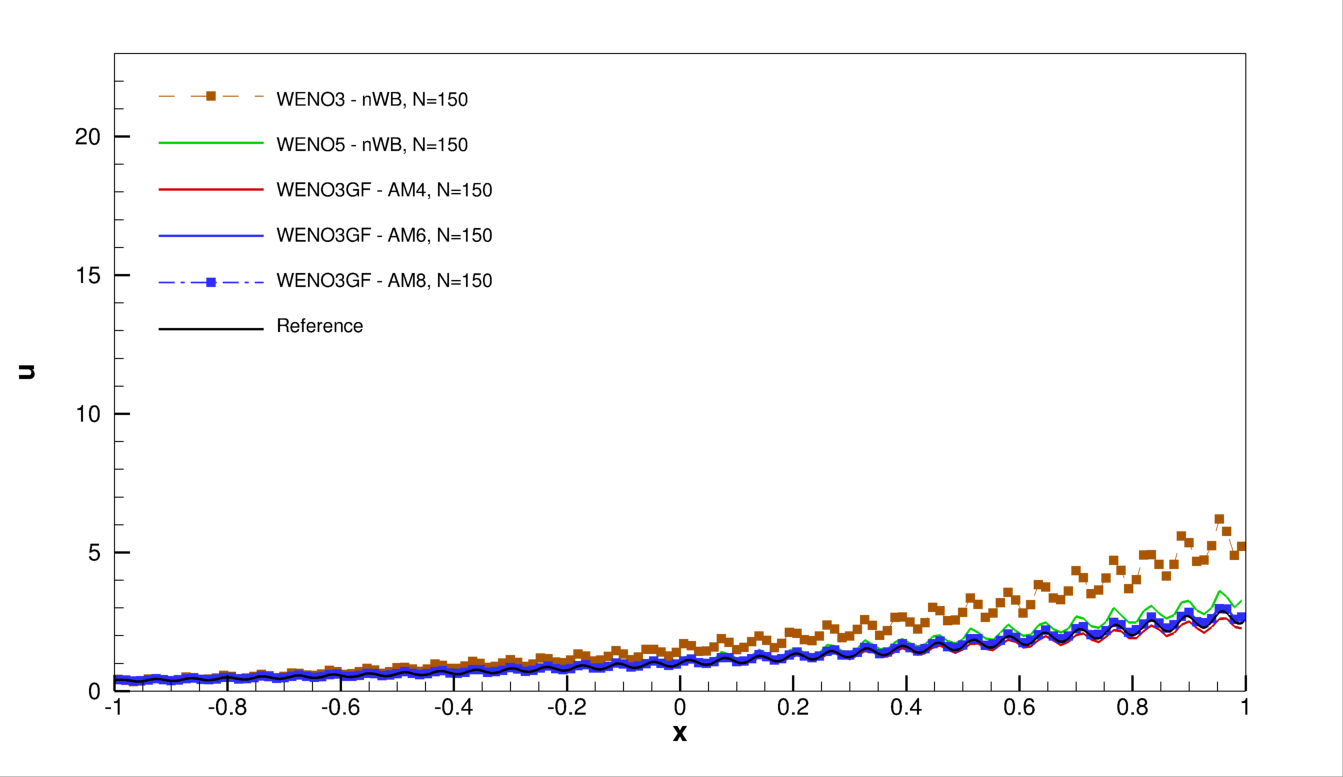}

\centering\includegraphics[width=0.5\textwidth]{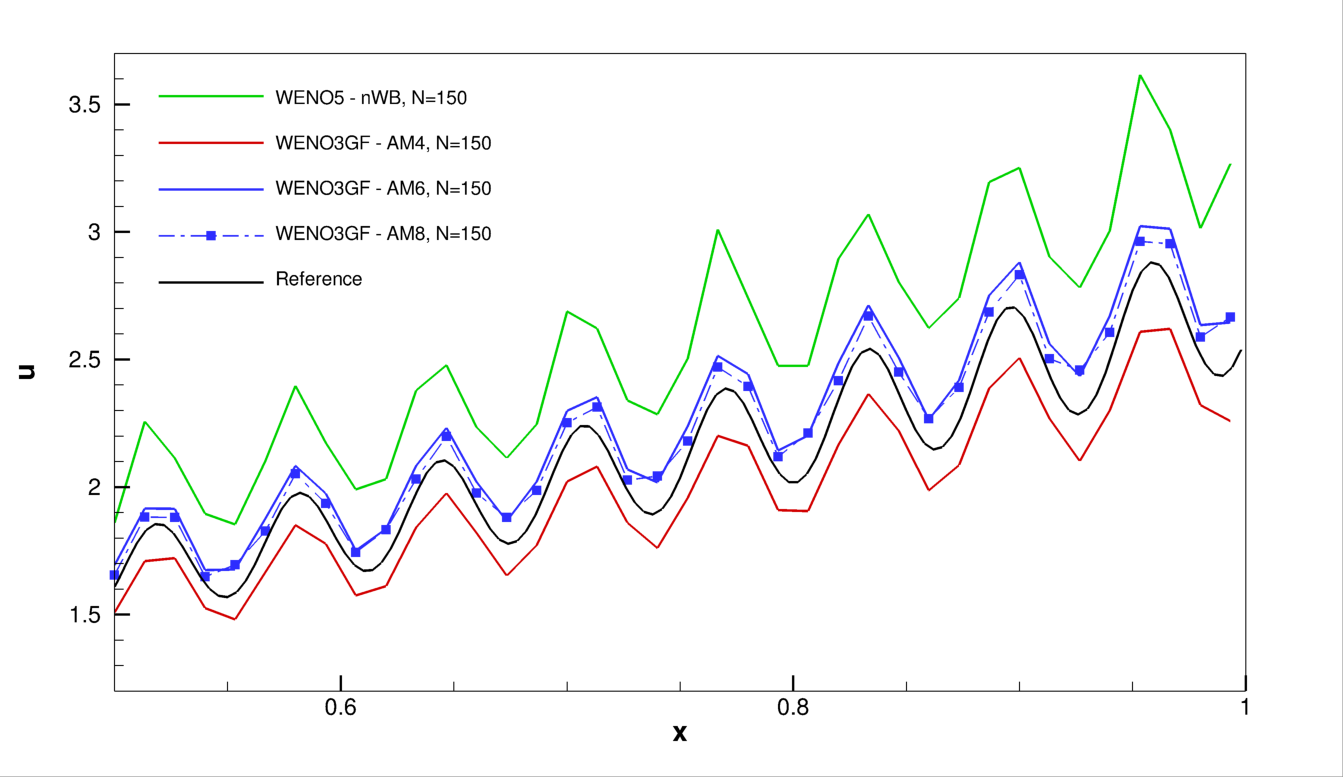}
\caption{Burgers' equation. Test \ref{ss:num_oscil}: highly oscillatory solution. Top:   discrete steady states with different schemes for  $N=100$ (left) and $N=150$ (Right). 
Bottom: zoom of the last oscillations for $N=150$. \label{Fig_test_oscill}}
\end{figure}

In order to see how the different numerical methods handle  perturbations in the next experiment we take initial data
$$
U_j = U^*_j + 0.2 \chi_{[-0.7, -0.5]}(x_j), \quad j=1, \dots, 100,
$$
where $\{ U^*_j \}$ is the discrete stationary solution reached by each method in the previous experiment with $100$ points, and $\chi_{[a,b]}$ represents the characteristic function of the interval $[a,b]$: see Figure \ref{fig-pBHo-1} (left). The differences between the numerical solutions at time $t = 0.7$ sec, and the underlying discrete stationary solutions are compared with a reference solution on the right in  Figure \ref{fig-pBHo-1}. It can be seen that  on a coarse mesh the non-WB   method
exaggerate considerably the evolution of the perturbation,
despite of the fact that the underlying steady state is its own
discrete stationary state. 

\begin{figure}
\includegraphics[width=0.5\textwidth]{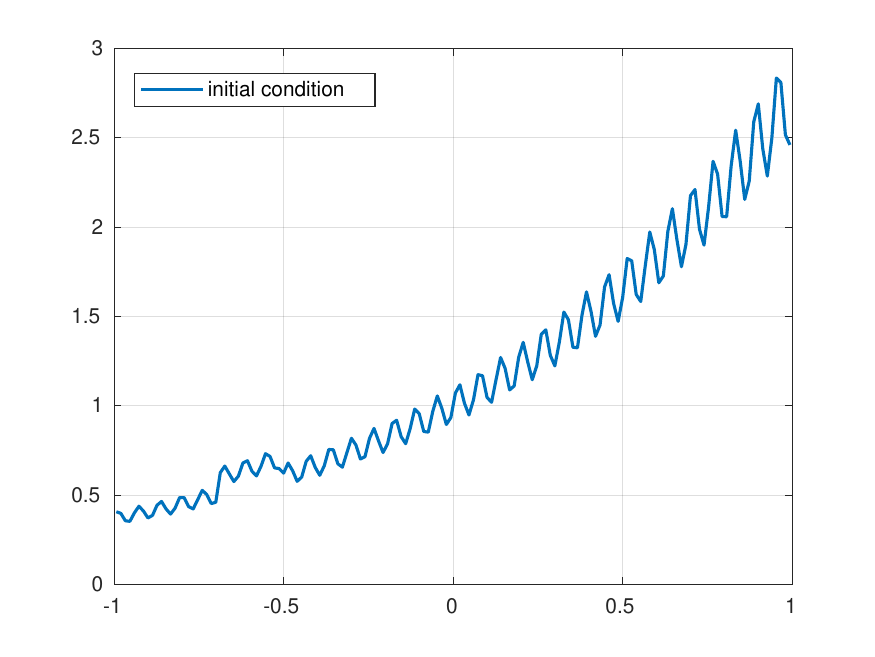}\includegraphics[width=0.5\textwidth]{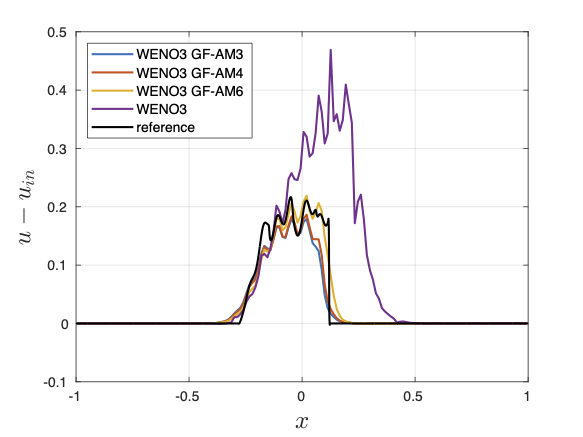}
\caption{Burgers' equation. Test \ref{ss:num_oscil}: highly oscillatory solution. Perturbed steady state ($N=150$). Left: initial solution. Right: evolution of the perturbation of different schemes. \label{fig-pBHo-1}}
\end{figure}

\subsubsection{Discontinuous data}

Let us consider now \eqref{testburgers} with a piecewise continuous function $H$:
\begin{equation} \label{H_disc_two}
H(x) = \begin{cases} 0.1 x & \text{ if $x \leq 0$;} \\
0.5+ x & \text{if $0 < x\leq 0.5$;} \\
0.9 +x & \text{otherwise.}
\end{cases}
\end{equation}
We consider again the interval $[-1,1]$ and we take as initial condition the stationary solution
\begin{equation}\label{stsolBHd}
u(x) = \begin{cases} 
e^{0.1 x} & \text{ if $x \leq 0$;} \\
e^{0.5 + x} & \text{ if $0< x\leq 0.5 $;} \\
e^{0.9 +x} & \text{otherwise.}
\end{cases}
\end{equation}

We deliberately chose two discontinuities on H. The first one, placed at $x = 0$ is aligned with a mesh point while the second one, placed at $x = 0.5$, is positioned at a cell face. 
Using the stationary solution as the initial condition over the interval $[-1,\; 1]$ on a mesh of $N=100$ points, we examine the impact on discontinuity preservation. As illustrated in Figure \ref{fig:two_disc}  WENO3GF-AMq , $q=4,6,8$ schemes successfully preserve both discontinuities irrespective of their location while WENO3-nWB smooths out the second discontinuity.

\begin{figure}
\centering\includegraphics[width=0.5\textwidth]{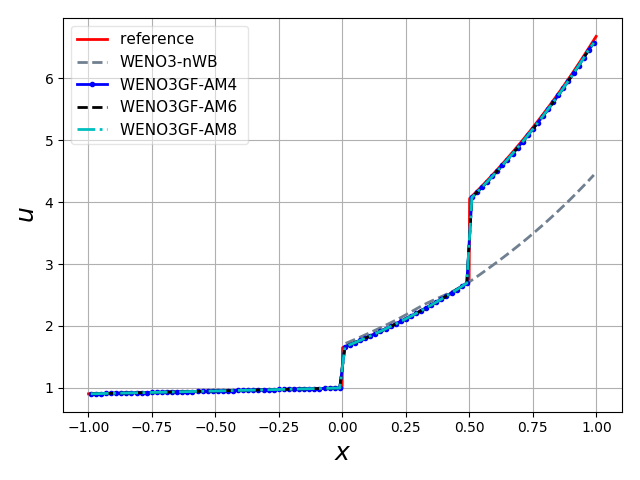}
\caption{Burgers' equation. Test \ref{ss:num_oscil}: discontinuous data. Preservation of two discontinuities at $t=0.2\;$ sec on a mesh of $N=100$ points}
\label{fig:two_disc}
\end{figure}

We now focus on the perturbation of the discontinuous solution. For clarity, we simplify the set up by retaining only a sigle discontinuity in $H$ so that function $H$ reads as:  

\begin{equation} \label{H_disc_one}
H(x) = \begin{cases} 0.1 x & \text{ if $x \leq 0$;} \\
0.9 +x & \text{otherwise;}
\end{cases}
\end{equation}
We consider again the interval $[-1,1]$ and we take as initial condition the stationary solution  with a small perturbation :
$$
\tilde u_0(x) = u(x) + 0.3 e^{-200(x+1/2)^2},
$$
with $ u(x) = e^{H}$, see also Figure \ref{fig:initial_disc}. The numerical solutions at times $t = 0.3$ and $0.5$ obtained with WENO3GF-AM$q$, $q = 4,6$  using a mesh with $N = 100$ and $300$ points are compared with  the numerical solutions of the  WENO3-nWB using the finer mesh,
and with a reference  obtained with the WENO5-nWB scheme with upwind fluxes and 5000 points. 
As it can be seen,   the non well balanced method
provides a wrong prediction of the interaction with the discontinuity,
while the  well balanced approach proposed here with the correction
to account for singularities allows to obtain an accurate result even on a very coarse mesh.

\begin{figure}
\centering\includegraphics[width=0.5\textwidth]{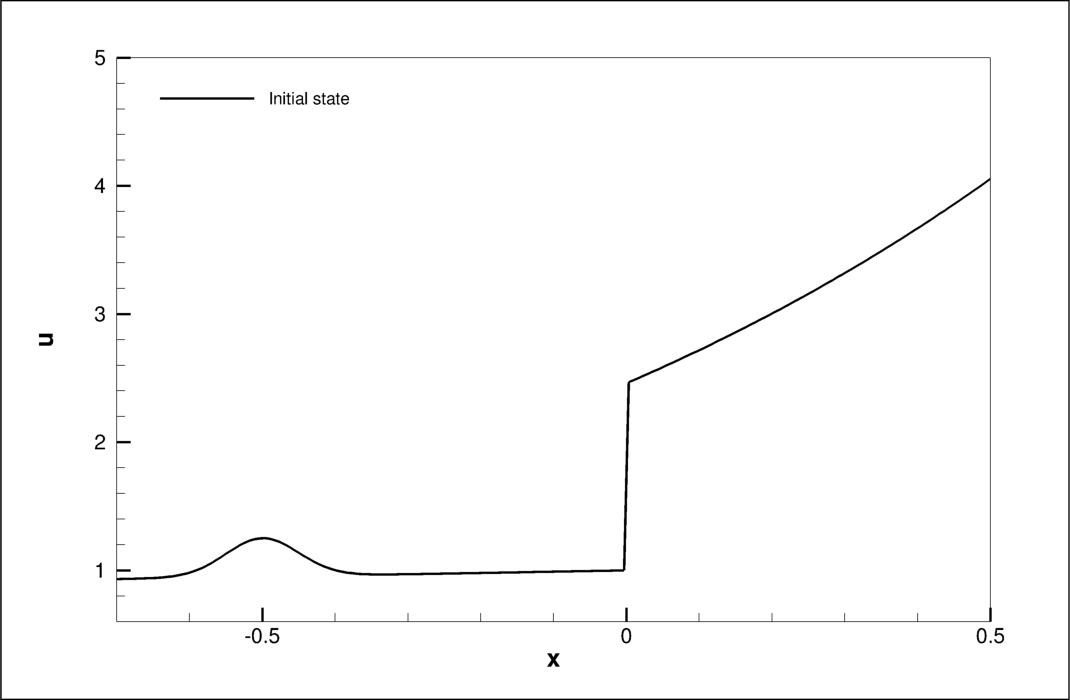}
\caption{Burgers' equation. Test \ref{ss:num_oscil}: discontinuous data. Perturbation of an isolated discontinuity: initial state.}
\label{fig:initial_disc}
\end{figure}

\begin{figure}
\begin{minipage}{0.5\textwidth}
\includegraphics[width=1.05\textwidth]{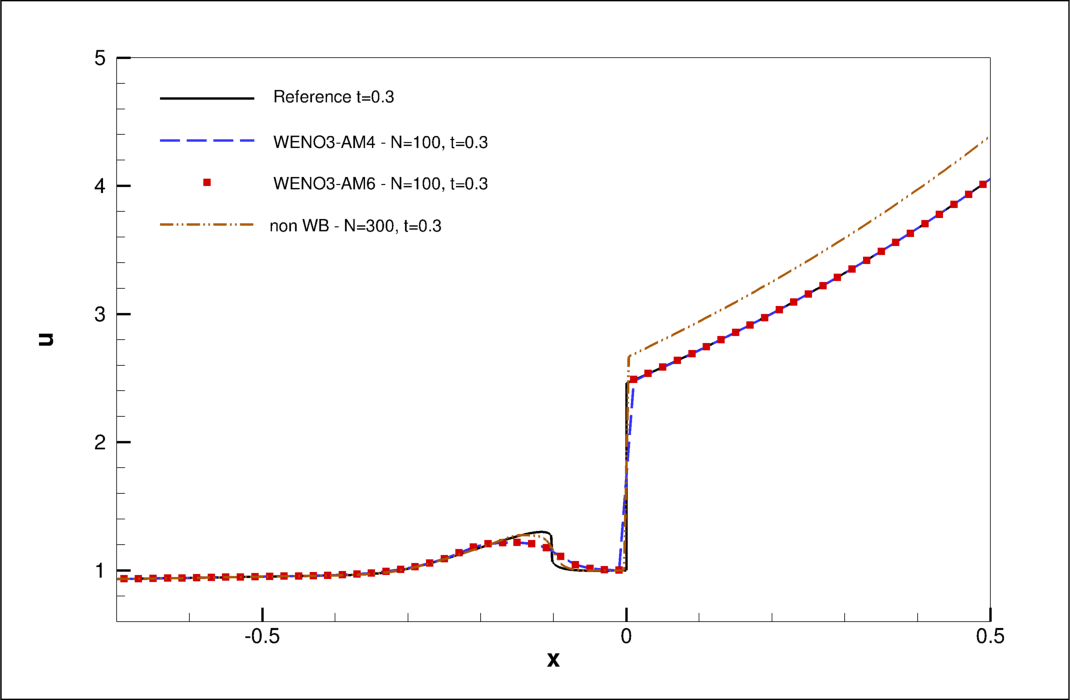}
\includegraphics[width=1.05\textwidth]{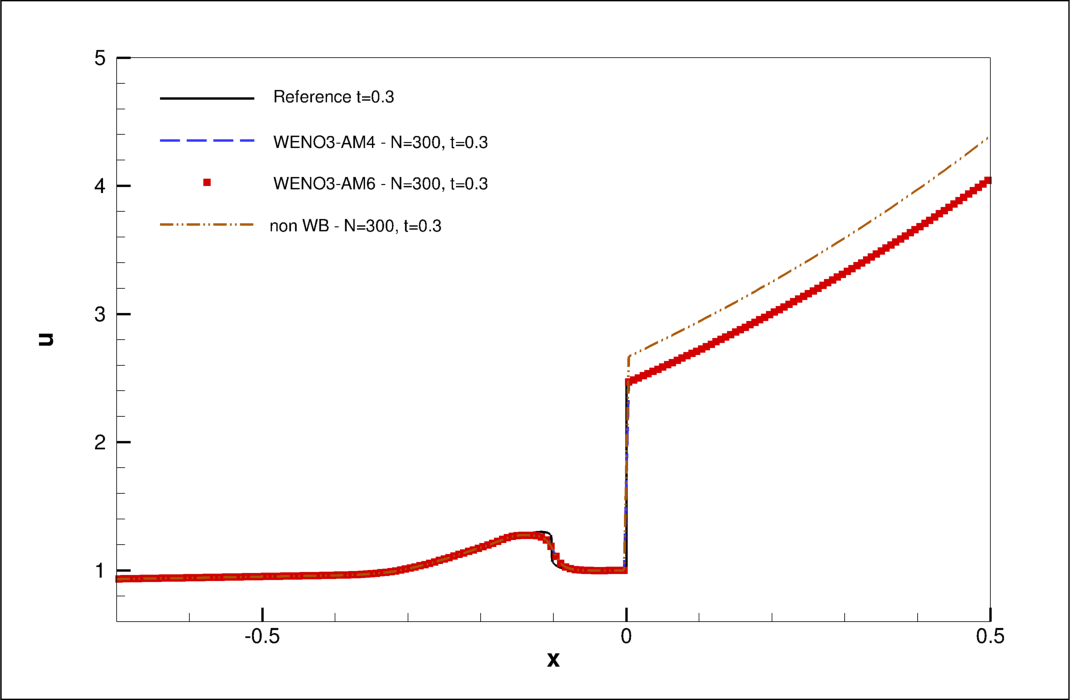}
\end{minipage}\hfill
\begin{minipage}{0.5\textwidth}
\includegraphics[width=1.05\textwidth]{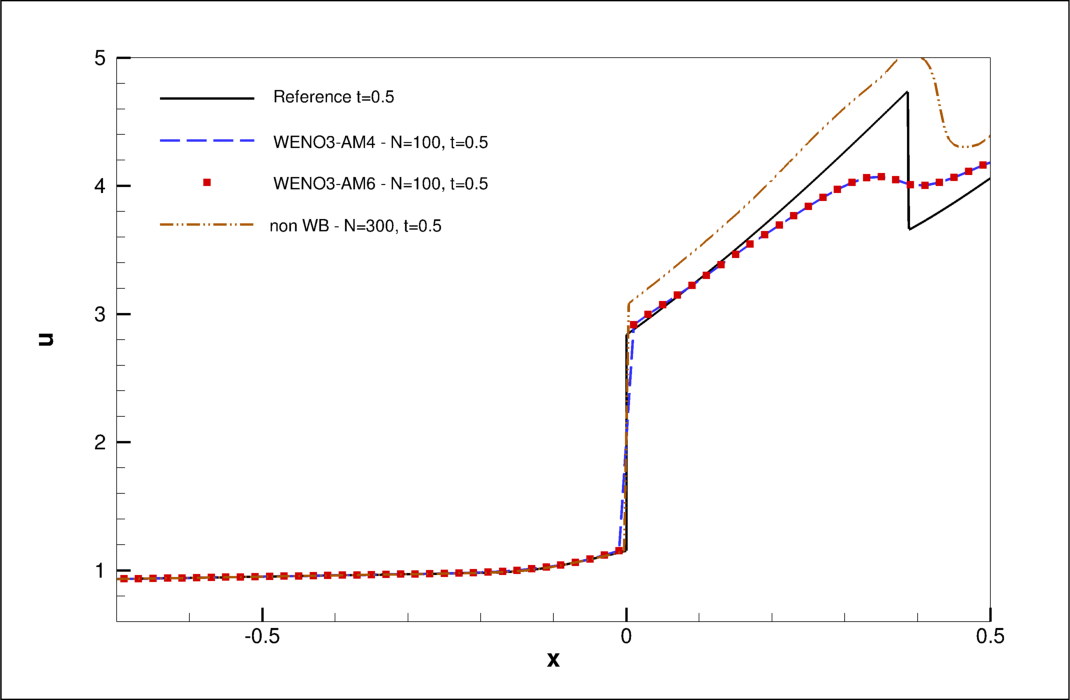}
\includegraphics[width=1.05\textwidth]{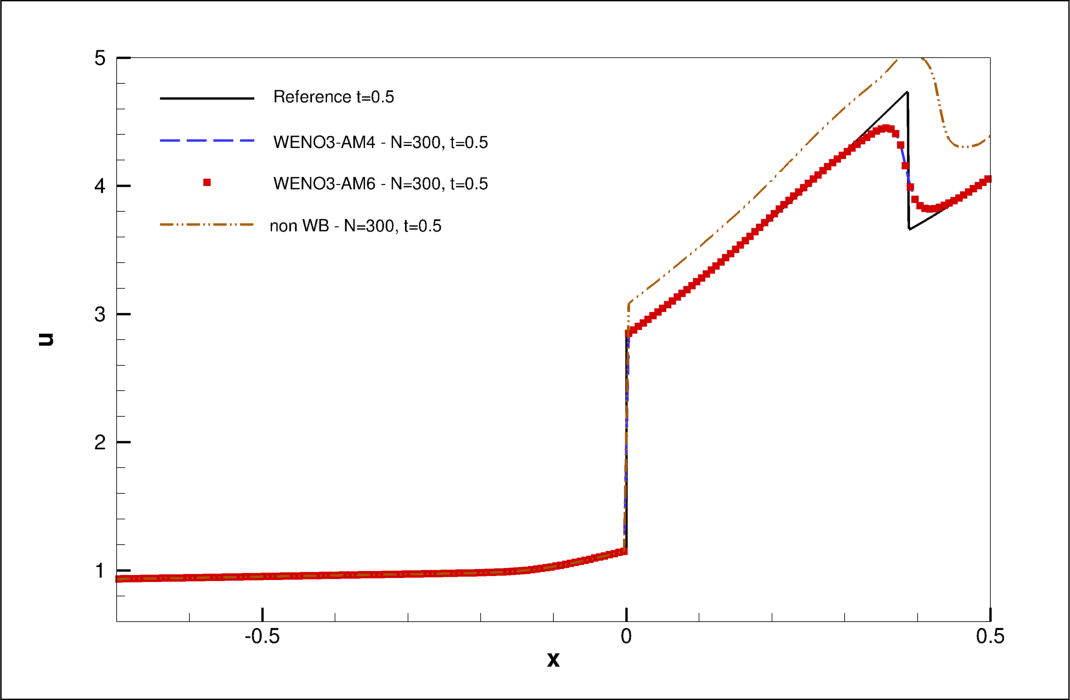}
\end{minipage}
\caption{Burgers' equation. Test \ref{ss:num_oscil}: discontinuous data. Perturbation of an isolated discontinuity: solutions obtained with WENO3GF-AM$q$, $q = 4,6$ at time $t=0.3$ (left) and $t=0.5$ (right) using $N=100$ (top) or $N = 300$ points (down) compared with a reference solution and with the numerical results given by WENO3-nWB using 300 points.}
\end{figure}

\subsection{Shallow water equations}
We now consider the shallow water system \eqref{sbl}-\eqref{sw}  and  verify the correct implementation and convergence of our methods using classical benchmarks that involve static and moving equilibria.


\subsubsection{Approximation of the lake at rest solution over a discontinuous bottom} \label{ss:lake_at_rest}
First, we focus on the lake at rest solution characterized by the initial data:
\begin{equation}\label{lake_at_rest_ss}
h(x,0) = 2-H(x), \;\; q(x,0) = 0
\end{equation}
with the bottom depth given by:
$$
H(x)= \begin{cases}
0.05 \sin(x - 12.5)  \exp\left(1 - (x - 12.5)^2\right), & x\leq 14 \\
0.05 \sin(x - 12.5)  \exp\left(1 - (x - 12.5)^2\right)-0.1, & x> 14.
\end{cases}
$$
in a computational domain of $[0, \; 25]$ m with subcritical inlet/outlet at the left and right boundary, respectively. The bathymetry consists of a sinusoidal bump with  a step of $0.1 \;$m  placed at $x=14 \; $m. Our goal in this test case is to demonstrate the exact preservation of water at rest solutions when we consider a polynomial approximation of H: see  section \ref{sec_lakeatrest}. Similar results are obtained when H is approximated using the selected multistep method. 
To do that, \eqref{lake_at_rest_ss} is taken as initial condition and the numerical methods are run until time $t = 2$ sec. The differences between the numerical solutions obtained with WENO3GF-AM$q$, $q = 4,6,8$ and the initial condition are shown in Table \ref{Table_SW_LAR_pres}: as it can be seen, the initial condition is preserved up to machine accuracy.

\begin{table}[h]
\centering
\caption{Shallow water equations. Test \ref{ss:lake_at_rest}: lake at rest solution. Errors in $L^1$ norm for WENO3GF-AM$q$, $q = 4,6,8$ at time $t=2$ sec.}
\label{Table_SW_LAR_pres}
\begin{tabular}{ |p{0.5cm}|c|c|c|  } 
\hline
N & \textbf{WENO3GF-AM4}   & \textbf{WENO3GF-AM6} & \textbf{WENO3GF-AM8}  \\
 \hline 
 25  & 1.776e-15 &   6.661e-15  &  2.042e-14\\
 50  & 1.476e-14 &   3.586e-14  &  3.907e-14\\
 100 & 2.953e-14 &   7.382e-14  &  8.387e-14\\
 200 & 6.586e-14 &   1.691e-13  &  2.054e-13\\
 400 & 1.426e-13 &   3.784e-13  &  4.344e-13\\
\hline
\end{tabular}
\end{table}

\subsubsection{Perturbation of the lake at rest solution over a discontinuous bottom} \label{ss:lake_at_rest_pert}

In this test case, we consider an initial condition which is obtained by adding a small perturbation to the initial solution considered in \ref{ss:lake_at_rest}. The perturbation is of size $\Delta h=10^{-4}$ and is added in $[7.5,\; 9.5]\;$m. Figure \ref{fig:lake_pert}, shows the initial perturbation on the left and the numerical solutions obtained with WENO3GF-AM$q$, $q = 4,6,8$ at time $t=1\;$sec using $N=100$ points. They are compared with the numerical solution given by WENO3-nWB and a reference solution computed with WENO7GF-AM8. As it can be seen, the GF methods capture correctly the evolution of the initial perturbation while standard WENO3 give errors that are several orders of magnitude bigger than the initial perturbation. 

\begin{figure}
\includegraphics[width=0.5\linewidth]{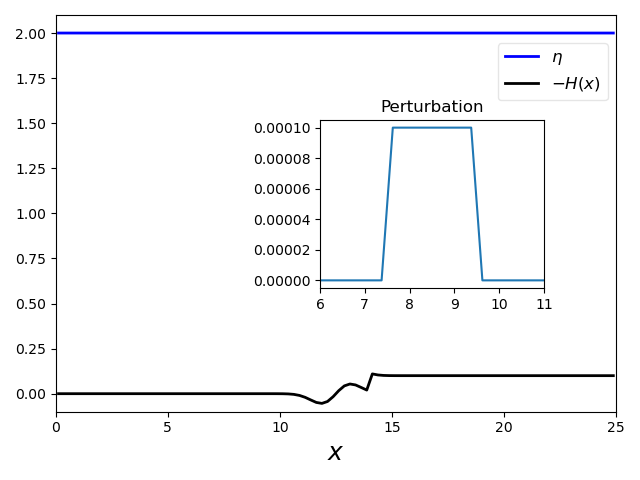}
\includegraphics[width=0.5\linewidth]{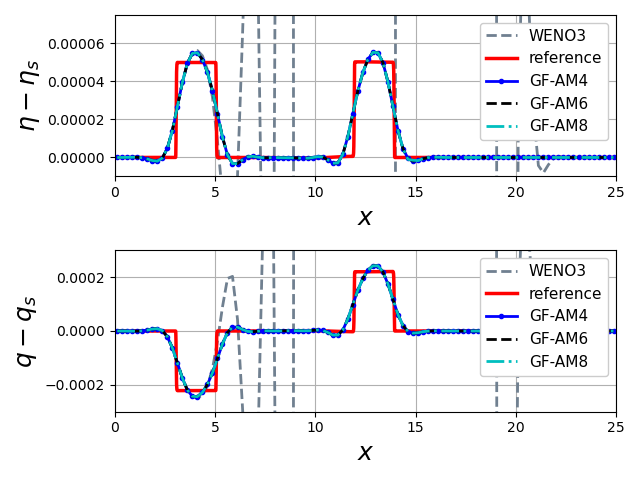}    
    \caption{Shallow water equations. Test \ref{ss:lake_at_rest_pert}: small perturbation on a lake at rest. Left: initial condition. Right: differences between the underlying stationary solution and the numerical solutions obtained with WENO3GF-AM$q$, $q = 4,6,8$ using $N = 100$ points at time $t=1\; sec$ compared with a reference solution and the numerical solutions given by WENO3-nWB. Top: $\eta$. Down: $q$. }
    \label{fig:lake_pert}
\end{figure}

We now change the size of the perturbation, and we place a Riemann problem instead.  $\Delta h$ is $1.0$ and it is added at the left of $x=12\;m$: see Fig. \ref{fig:lake_riemann} on the left. The numerical solutions for different schemes, after $t=1 \; sec$, are plotted on the right. Again, the evolution of the perturbation is better captured by the WB methods as one may expect, especially due to the  presence of the discontinuity in $H$.

\begin{figure}
\includegraphics[width=0.5\linewidth]{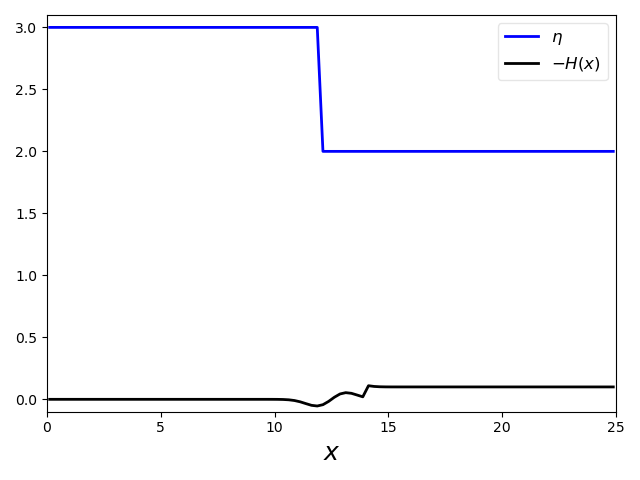}
\includegraphics[width=0.5\linewidth]{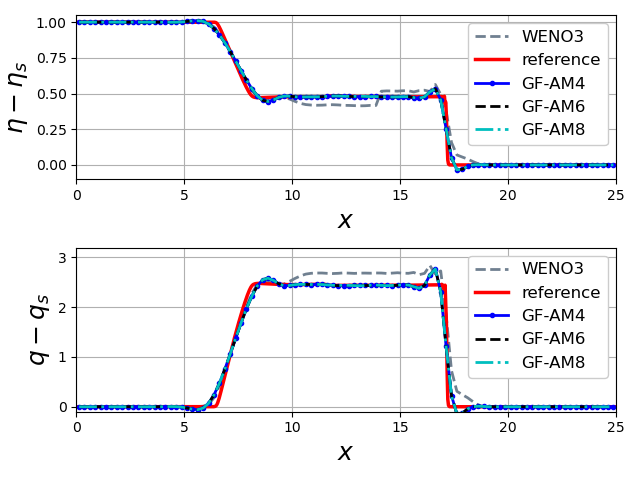}    
    \caption{Shallow water equations. Test \ref{ss:lake_at_rest_pert}: big perturbation on a lake at rest. Left: initial condition. Right: differences between the underlying stationary solution and the numerical solutions obtained with WENO3GF-AM$q$, $q = 4,6,8$ using $N = 100$ points at time $t=1\; sec$ compared with a reference solution and the numerical solutions given by WENO3-nWB. Top: $\eta$. Down: $q$.  }
    \label{fig:lake_riemann}
\end{figure}
\noindent


\subsubsection{Approximation of subcritical and supercritical stationary solutions}\label{ss:num_subcr}

We consider the shallow water system \eqref{sbl}-\eqref{sw} with the bottom depth given by:
$$
H(x) = 0.05 \sin(x - 12.5)  \exp\left(1 - (x - 12.5)^2\right). 
$$
As initial condition, we take the subcritical stationary solution characterized by
$$
q^*(0) = 4.42, \quad h^*(0) = 2,$$
whose values at the points are computed using the relations \eqref{ssSW}. The numerical methods then run  until a discrete steady state is reached  up to machine accuracy. The boundary conditions used are: 
$$
q(0,t) =4.42, \;\; h(25,t)=2.
$$
Figure \ref{fig:steady} (right) shows the discrete stationary solution reached by WENO3GF-AM4 using $N = 100$ points. 
\begin{figure}
\includegraphics[width=0.5\linewidth]{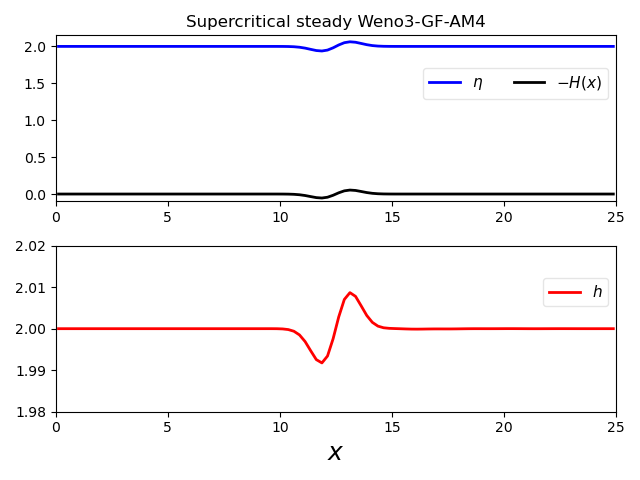}\includegraphics[width=0.5\linewidth]{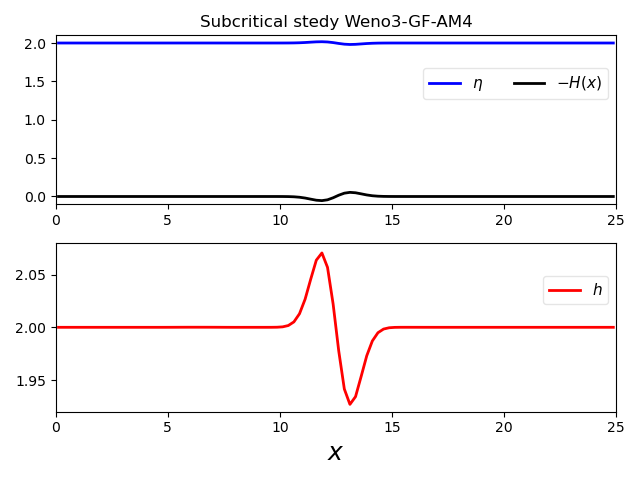}    
    \caption{Shallow water equations. Test \ref{ss:num_subcr}: discrete steady state obtained with WENO3GF-AM4 using $N = 100$ points. Left: supercritical case. Right: subcritical case. Top: free surface and bottom. Down: $h$.}
    \label{fig:steady}
\end{figure}

Table \ref{TableSubcr} shows the errors in $L^1$ norm and the empirical order of convergence of the different methods using WENO3GF-AM$q$, $q = 4,6,8$.  For reference, the errors corresponding to WENO3-nWB are also shown. We should mention that, as observed in the case of Test \ref{ss:num_scalar_steady}, the empirical orders and errors for WENO$p$GF-AB$q$ are independent of $p$, which is why we are only presenting the table for WENO3. As it can be seen, we confirm that, also for a system of equations, the order of accuracy for stationary solutions is the one of the multi-step method.   

\begin{table}[h]
\centering
\caption{Shallow water equation. Test \ref{ss:num_subcr}: subcritical steady state. Errors in $L^1$ norm and convergence rate for WENO3GF-AM$q$, $q=4,5,8$ (left) and WENO3-nWB (right).}
\label{TableSubcr}
\begin{minipage}{0.5\textwidth}
\centering
\begin{tabular}{ |p{0.5cm}|c@{\hskip 5pt}c|c@{\hskip 5pt}c|c@{\hskip 5pt}c|  } 
\hline
 & \multicolumn{2}{|c|}{4 steps} & \multicolumn{2}{|c|}{6 steps} & \multicolumn{2}{|c|}{8 steps} \\
 \hline
 N & Error  & Order & Error & Order & Error & Order \\
 \hline
 \multicolumn{7}{|c|}{\textbf{WENO3GF - AM}} \\
 \hline
 25  & 2.507e-03 & -    & 3.510e-03   & -    &  4.861e-03  & -\\
 50  & 1.032e-04 & 4.6  & 3.700e-04   & 3.2  &  8.815e-04  & 2.5\\
 100 & 1.209e-05 & 3.1  & 1.098e-05   & 5.0  &  1.481e-05  & 5.9\\
 200 & 8.110e-07 & 3.9  & 2.298e-07   & 5.6  &  9.337e-08  & 7.3\\
 400 & 5.279e-08 & 4.0  & 3.817e-09   & 5.9  &  4.181e-10  & 7.8\\
 \hline
 \multicolumn{7}{|c|}{\textbf{WENO3GF - AB}} \\
 \hline
 25  & 1.033e-02 & -    &  7.332e-02  & -    &  1.555e-01   & -\\
 50  & 1.990e-03 & 2.4  &  6.356e-03  & 3.5  &  1.837e-02   & 3.0\\
 100 & 1.326e-04 & 3.9  &  1.583e-03  & 2.0  &  4.255e-04   & 5.4\\
 200 & 1.054e-05 & 3.7  &  2.940e-05  & 5.7  &  2.851e-06   & 7.2\\
 400 & 6.937e-07 & 3.9  &  4.185e-09  & 6.3  &  8.015e-08   & 8.0\\
\hline
\end{tabular}
\end{minipage}
\hfill
\begin{minipage}{0.3\textwidth}
\centering
\begin{tabular}{ |p{0.5cm}|c@{\hskip 5pt}c| } 
\hline
\multicolumn{3}{|c|}{\textbf{WENO3-nWB}} \\ 
\hline
 N & Error  & Order \\
 \hline
 25   & 2.157e-01  & -  \\
 50   & 1.349e-02  & 4.0  \\
 100  & 1.840e-03  & 2.9  \\
 200  & 2.308e-04  & 3.0  \\
 400  & 2.882e-05  & 3.0  \\
\hline
\end{tabular}
\end{minipage}
\end{table}

Next, we consider as a safety check the same problem  with a reversed flow direction.
Note that we do not change the ODE weights in the scheme to integrate the source term.
Table \ref{TableSubcr_re} confirms that, despite their non-symmetric structure, the GF schemes obtained with both multi-step methods are able to achieve the expected order of accuracy.   \\ 

The same   tests  are performed  for the supercritical stationary solution characterized by 
$$
h^*(0) =2, \quad q^*(0)=24.
$$
The left picture on Figure \ref{fig:steady}  shows the discrete stationary solution reached by WENO3GF-AM4 using $N = 100$ points. 
The convergence results are reported on  table \ref{TableSuper}, and show the same behaviour also observed for the subcritical case, while Table \ref{TableSuper_re} reports the same for the reversed supercritical flow confirming once again that the results are independent of the flow direction.

\begin{table}[h]
\centering
\caption{Shallow water equation. Test \ref{ss:num_subcr}: subcritical reversed steady state. Errors in $L^1$ norm and convergence rate for WENO3GF-AM$q$, $q=4,5,8$ (left) and WENO3-nWB (right).
\label{TableSubcr_re}}
\begin{minipage}{0.5\textwidth}
\centering
\begin{tabular}{ |p{0.5cm}|c@{\hskip 5pt}c|c@{\hskip 5pt}c|c@{\hskip 5pt}c|  } 
\hline
 & \multicolumn{2}{|c|}{4 steps} & \multicolumn{2}{|c|}{6 steps} & \multicolumn{2}{|c|}{8 steps} \\
 \hline
 \multicolumn{7}{|c|}{\textbf{WENO3GF - AM}} \\
  \hline
 N & Error  & Order & Error & Order & Error & Order \\
 \hline
 50  & 1.033e-04 & -    & 3.633e-04   & -    &  8.636e-04  & -\\
 100 & 1.217e-05 & 3.1  & 1.098e-05   & 5.0  &  1.470e-05  & 5.9\\
 200 & 8.151e-07 & 3.9  & 2.289e-07   & 5.6  &  9.326e-08  & 7.3\\
 400 & 5.317e-08 & 3.9  & 3.809e-09   & 5.9  &  4.153e-10  & 7.8\\
 \hline
 \multicolumn{7}{|c|}{\textbf{WENO3GF - AB}} \\
 \hline
 50  & 1.119e-03 & -    &  5.626e-03  & -    &  1.805e-02   & -  \\
 100 & 1.329e-04 & 3.8  &  2.327e-04  & 5.0  &  4.228e-04   & 5.4\\
 200 & 1.056e-05 & 3.7  &  4.745e-06  & 5.6  &  2.849e-06   & 7.2\\
 400 & 6.989e-07 & 3.9  &  8.351e-08  & 5.8  &  1.300e-08   & 7.8 \\
\hline
 \end{tabular}
\end{minipage}
\hfill
\begin{minipage}{0.3\textwidth}
\centering
\begin{tabular}{ |p{0.5cm}|c@{\hskip 5pt}c| } 
\hline
\multicolumn{3}{|c|}{\textbf{WENO3-nWB}} \\ 
\hline
 N & Error  & Order \\
 \hline
 50   & 4.594e-02  & -  \\
 100  & 7.163e-03  & 2.7  \\
 200  & 1.620e-03  & 2.1  \\
 400  & 2.380e-04  & 2.8  \\
\hline
\end{tabular}
\end{minipage}
\end{table}

\noindent

\begin{table}[h]
\centering
\caption{Test \ref{ss:num_subcr}: Supercritical flow. Errors in $L^1$ norm and convergence rate for WENO3 \label{TableSuper}}
\begin{minipage}{0.5\textwidth}
\centering
\begin{tabular}{ |p{0.5cm}|c@{\hskip 5pt}c|c@{\hskip 5pt}c|c@{\hskip 5pt}c|  } 
\hline
 & \multicolumn{2}{|c|}{4 steps} & \multicolumn{2}{|c|}{6 steps} & \multicolumn{2}{|c|}{8 steps} \\
 \hline
 N & Error  & Order & Error & Order & Error & Order \\
 \hline
 \multicolumn{7}{|c|}{\textbf{WENO3GF - AM}} \\
 \hline
 25  & 4.270e-04 & -    & 6.847e-04   & -    &  1.049e-03  & -\\
 50  & 4.181e-05 & 3.4  & 1.296e-04   & 2.4  &  3.310e-04  & 1.6\\
 100 & 4.113e-06 & 3.3  & 4.210e-06   & 4.9  &  5.289e-06  & 6.0\\
 200 & 2.749e-07 & 3.9  & 8.458e-08   & 5.6  &  3.524e-08  & 7.2\\
 400 & 1.778e-08 & 4.0  & 1.411e-09   & 5.9  &  1.628e-10  & 7.8\\
 \hline
 \multicolumn{7}{|c|}{\textbf{WENO3GF - AB}} \\
 \hline
 25  & 2.166e-03 & -    &  5.047e-03  & -    &  1.226e-02  & -\\
 50  & 6.085e-04 & 1.8  &  2.124e-03  & 1.2  &  7.079e-03  & 0.7\\
 100 & 5.089e-05 & 3.6  &  8.074e-05  & 4.7  &  1.394e-04  & 5.7\\
 200 & 3.646e-06 & 3.8  &  1.777e-06  & 5.5  &  1.092e-06  & 7.0\\
 400 & 2.336e-07 & 4.0  &  3.072e-08  & 5.9  &  4.928e-09  & 7.8\\
\hline
\end{tabular}
\end{minipage}
\hfill
\begin{minipage}{0.3\textwidth}
\centering
\begin{tabular}{ |p{0.5cm}|c@{\hskip 5pt}c| } 
\hline
\multicolumn{3}{|c|}{\textbf{WENO3-nWB}} \\ 
\hline
 N & Error  & Order \\
 \hline
 25   & 1.678e-01  & -  \\
 50   & 2.723e-03  & 2.6  \\
 100  & 4.562e-04  & 2.6  \\
 200  & 8.543e-05  & 2.4  \\
 400  & 1.108e-05  & 2.9  \\
\hline
\end{tabular}
\end{minipage}
\end{table}

\begin{table}[h]
\centering
\caption{Shallow water equation. Test \ref{ss:num_subcr}: supercritical reversed steady state. Errors in $L^1$ norm and convergence rate for WENO3GF-AM$q$, $q=4,5,8$ (left) and WENO3-nWB (right).
\label{TableSuper_re}}
\begin{minipage}{0.5\textwidth}
\centering
\begin{tabular}{ |p{0.5cm}|c@{\hskip 5pt}c|c@{\hskip 5pt}c|c@{\hskip 5pt}c|  } 
\hline
 & \multicolumn{2}{|c|}{4 steps} & \multicolumn{2}{|c|}{6 steps} & \multicolumn{2}{|c|}{8 steps} \\
 \hline
 \multicolumn{7}{|c|}{\textbf{WENO3GF - AM}} \\
  \hline
 N & Error  & Order & Error & Order & Error & Order \\
 \hline
 50  & 4.187e-05 & -    & 1.303e-04   & -    &  3.314e-04  & -\\
 100 & 4.128e-06 & 3.4  & 4.212e-06   & 5.0  &  5.292e-06  & 6.0\\
 200 & 2.759e-07 & 3.9  & 8.455e-08   & 5.6  &  3.524e-08  & 7.2\\
 400 & 1.784e-08 & 4.0  & 1.408e-09   & 5.9  &  1.618e-10  & 7.8\\
 \hline
 \multicolumn{7}{|c|}{\textbf{WENO3GF - AB}} \\
 \hline
 50  & 6.097e-04 & -    &  2.125e-03  & -    &  7.092e-03   & -  \\
 100 & 5.090e-05 & 3.6  &  8.095e-05  & 4.7  &  1.395e-04   & 5.6\\
 200 & 3.657e-06 & 3.8  &  1.776e-06  & 5.5  &  1.092e-06   & 7.0\\
 400 & 2.344e-07 & 4.0  &  3.070e-08  & 5.9  &  4.928e-09   & 7.8 \\
\hline
 \end{tabular}
\end{minipage}
\hfill
\begin{minipage}{0.3\textwidth}
\centering
\begin{tabular}{ |p{0.5cm}|c@{\hskip 5pt}c| } 
\hline
\multicolumn{3}{|c|}{\textbf{WENO3-nWB}} \\ 
\hline
 N & Error  & Order \\
 \hline
 50   & 2.397e-03  & -  \\
 100  & 2.996e-04  & 3.0  \\
 200  & 3.879e-05  & 2.9  \\
 400  & 4.910e-06  & 3.0  \\
\hline
\end{tabular}
\end{minipage}
\end{table}

\subsubsection{Perturbations subcritical and supercritical  stationary solutions} \label{ss:subcritical_pert}

As in the case of Test \ref{ss:lake_at_rest_pert}, we consider now an initial condition that represents a perturbation of a moving stationary solution.
For brevity we consider only the subcritical case of  section \ref{ss:num_subcr}. 
The results for the supercritical are very similar.

First, we consider the evolution of a small perturbation  $\Delta h=10^{-4}\; $m,
which is is added to the depth in the subdomain $[7.5,\; 9.5]\;m$.
The left picture on figure \ref{fig:sub_pert} shows the numerical results  obtained with WENO3GF-AM$q$ using a mesh of $N =100$ points at time $t = 1$ sec.
For comparison we report the results of the non-well balanced WENO3 scheme, and a reference solution obtained on 500 points with the WENO7GF-AM8 scheme.
We can see that on this coarse mesh the well balanced formulations allow to nicely
follow the evolution of such a small perturbation.\\

For completeness we also consider a perturbation of $\Delta h=1\; $m,
which is added for all points  to the left of $x = 12\,$m. For the smooth bathymetry used in this case all schemes provide a similar evolution of the solution.

\begin{figure}
\includegraphics[width=0.5\linewidth]{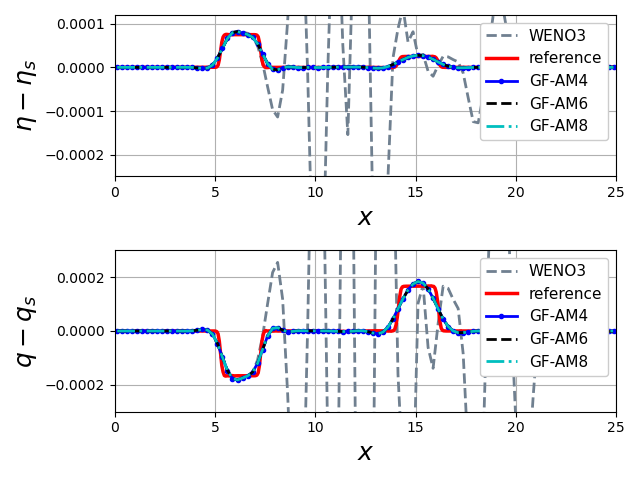}
\includegraphics[width=0.5\linewidth]{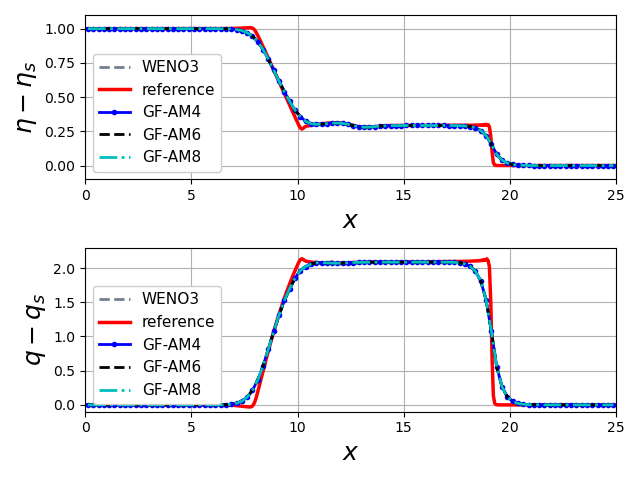}    
    \caption{Shallow water equations. Test \ref{ss:subcritical_pert}: perturbation of a subcritical steady state. Left: small perturbation. Right: big perturbation. Differences between the underlying stationary solution and the numerical solutions obtained with WENO3GF-AM$q$, $q = 4,6,8$ using $N = 100$ points at time $t=1\; sec$ compared with a reference solution and the numerical solutions given by WENO3-nWB. Top row: $\eta$. Bottom row: $q$.}
    \label{fig:sub_pert}
\end{figure}
\noindent

%

\noindent
\subsubsection{Approximation of a transcritical stationary solution} \label{ss:num_trans}

For completeness we also report the results obtained for a transcritical case.
To begin with we compute the steady states using the AB ODE solver of different accuracies and perform a grid convergence  to the  analytical solution
obtained solving the exact cubic relation  associated to (8). 
The results are reported on table \ref{TableSuper_re} in which we compare
the accuracy of the steady  WENO5FG-AB for different accuracies, with those
of the standard non well balanced upwind WENO5. 
Note that we do not modify the flux splitting to account for the 
presence of the sonic point. 
The tables show that on the mesh sizes considered also for this case the errors obtained using the global flux quadrature formulation are of orders of magnitude lower that those of the upwind scheme.

\begin{table}[h]
\centering
\caption{Shallow water equation. Test \ref{ss:num_trans}: transcritical flow over a smooth bump. Errors in $L^1$ norm and convergence rate for WENO5GF-AB$q$, $q=4,6,8$ (left) and WENO5-nWB (right).
\label{TableSuper_re}}
\begin{minipage}{0.5\textwidth}
\centering
\begin{tabular}{ |p{0.5cm}|c@{\hskip 5pt}c|c@{\hskip 5pt}c|c@{\hskip 5pt}c|  } 
\hline
 & \multicolumn{2}{|c|}{4 steps} & \multicolumn{2}{|c|}{6 steps} & \multicolumn{2}{|c|}{8 steps} \\
 \hline
 \multicolumn{7}{|c|}{\textbf{WENO5GF - AB}} \\
  \hline
 N & Error  & Order & Error & Order & Error & Order \\
 \hline
 50  & 0.0991049 & -    & 0.2174215   & -    &  0.2381622  & -\\
 100 & 0.0066493 & 3.91  & 0.0088290   & 4.64  &  0.0062804  & 5.26\\
 200 & 0.0004984 & 3.75  & 0.0001740   & 5.68  &  4.317e-05  & 7.21\\
 400 & 4.483e-05 & 3.49  & 3.375e-06   & 5.71  &  1.101e-06  & 5.31\\
 \hline
 \end{tabular}
\end{minipage}
\hfill
\begin{minipage}{0.3\textwidth}
\centering
\begin{tabular}{ |p{0.5cm}|c@{\hskip 5pt}c| } 
\hline
\multicolumn{3}{|c|}{\textbf{WENO5-nWB}} \\ 
\hline
 N & Error  & Order \\
 \hline
 50   & 0.2305202  & -  \\
 100  & 0.0913807  & 1.33  \\
 200  & 0.0194528  & 2.24  \\
 400  & 0.0004357  & 5.50  \\
\hline
\end{tabular}
\end{minipage}
\end{table}

\subsubsection{Perturbation of a transcritical stationary solution} \label{ss:trans_pert}

We also consider the evolution of perturbations of the smooth transcritical flow of the previous section. For each  scheme we consider the perturbation 
of the well prepared  data corresponding to its discrete steady state.
In the case of the GF-AB schemes this is obtained directly using the ODE solver. We add to this data a   small perturbation of $\Delta h= 10^{-4}$ m. Figure \ref{fig:trans_pert} shows on the left the discrete steady solution obtained using  the ODE AB and on the right the numerical results obtained with WENO5-AB$q$,$q = 4,6,8$. We used a mesh of $N=100$ points at time $t=0.7$ sec. We also report the results of the WENO5-nWB and a reference solution obtained with $N=1000$ points and using the WENO7-AM8 scheme. Like before we confirm that the well balanced formulation allow to nicely follow the evolution of the small pertrubation.     

\begin{figure}
\includegraphics[width=0.5\linewidth]{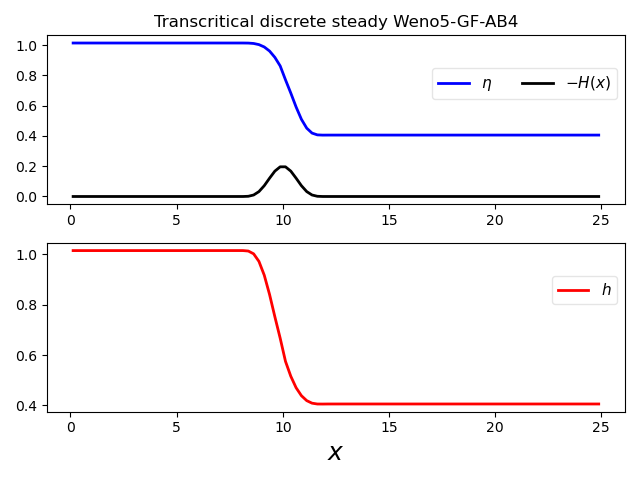}
\includegraphics[width=0.5\linewidth]{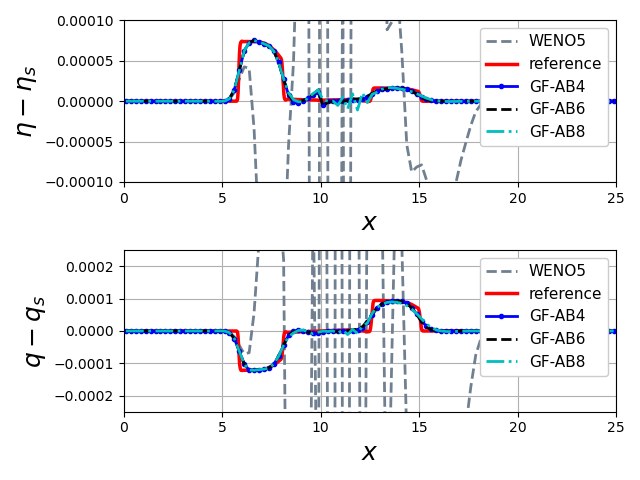}
    \caption{Shallow water equations. Test \ref{ss:trans_pert}: perturbation of a transcritical discrete steady state. Left: Discrete steady solution. Right: Differences between the underlying discrete stationary solution and the numerical solutions obtained with WENO5GF-AB$q$, $q = 4,6,8$ using $N = 100$ points at time $t=0.7\; sec$ compared with a reference solution and the numerical solutions given by WENO5-nWB. Top row: $\eta$. Bottom row: $q$.}
    \label{fig:trans_pert}
\end{figure}
\noindent


\subsubsection{Moving equilibria with bathymetry and friction}

\paragraph{Analytical solutions with $\kappa h q|q|$ friction.}  
We consider here  \eqref{sw1} with the  friction closure   
$$
\kappa = k h |q|.
$$
In this case, assuming that the flow runs left to right, and  setting $\eta=h-H$, the steady ODE \eqref{ssSW1} can be recast as   
$$
- C_1^2h^{-2} h_x  + gh \eta_x  = - k C_1^2 h
$$
which, dividing through by $h$ gives
\begin{equation}\label{eq:f-ode}
- C_1^2h^{-3} h_x  + g \eta_x  = - k C_1^2.
\end{equation}
This equation can be integrated to obtain the equalities
$$
q = C_1, \quad \frac{1}{2} \dfrac{q^2}{h^2} + gh - gH - k q^2 x = C_2.
$$
Given $C_i$, $i=1,2$ the point values of a stationary solution can be then computed by solving a cubic equation like in the case without friction.

To obain a fully analytical expression we proceed in a slightly different manner. We impose a prescribed profile $z(x)$  to the free surface elevation $\eta$, and then compute $h$   by solving \eqref{eq:f-ode}. A compatible definition of the depth at rest  $H$ is obtained as  $H = h - z$.
If the boundary condition
$$
h(x_0) = h_0, \quad q(x_0) = q_0$$
is imposed, the corresponding stationary solution is given by
\begin{equation}\label{eq:f-solution}
\begin{aligned}
   q^{\textsf{e}}(x) =& q_0,\\
    h^{\textsf{e}}(x) = &   \dfrac{h_0}{\sqrt{ 1 + \dfrac{2h_0^2g}{q_0^2} ( z(x_0) -z(x) )   - 2 k h_0^2  (x-x_0)} }, & \\
    H(x) = & h^{\textsf{e}}(x) - z^{\textsf{e}}(x) .
\end{aligned}
\end{equation}
We now consider the free surface profiles 
\begin{equation}\label{eq:zx}
    z(x)= \mathsf{a}-\mathsf{b}*\dfrac{\mathsf{c}x e^{\cos(4\pi x)}  - 1/e }{ e-1/e}
\end{equation}
One can show that a supercritical solution is obtained for $$(h_0,q_0,\mathsf{a},\mathsf{b},\mathsf{c},k)=(1,3/2,5/2,1/2,2,0.3),$$ while a subcritical one is obtained for
$$(h_0,q_0,\mathsf{a},\mathsf{b},\mathsf{c},k)=(1,0.3,5/2,1/4,1/2,1/2).$$
The two solutions are plotted for completeness in Figure \ref{fig:frictsols}.

\begin{figure}
\includegraphics[width=0.5\linewidth]{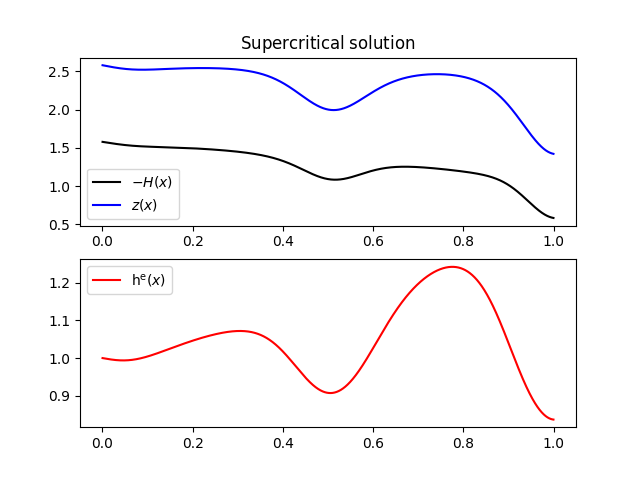}\includegraphics[width=0.5\linewidth]{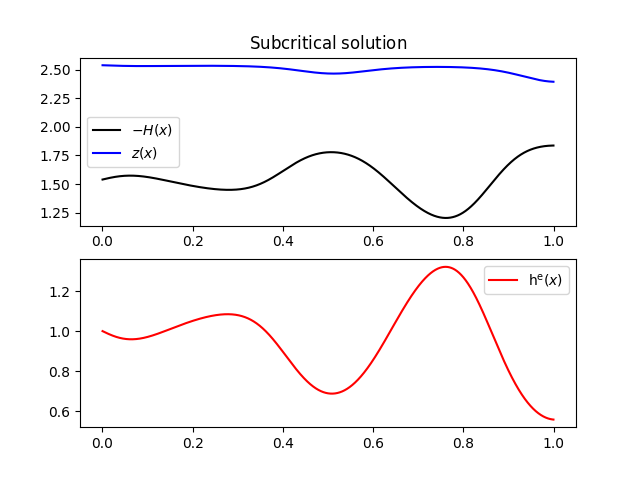}    
    \caption{Shallow water equations with friction law $\kappa h q|q|$. Supercritical  (left) and subcritical (right) exact solutions. Top: free surface and bottom. Down: $h$. }
    \label{fig:frictsols}
\end{figure}

We start by computing the  grid convergence of the discrete steady state to the exact one.  In Table \ref{tab:firct-super} the $L^1$ errors with respect to the analytical solution for WENO3GF-AM$q$, $q = 4,6,8$ are compared with the corresponding to the standard WENO3 and WENO5 schemes.  Similar results are obtained for WENO5GF-AM$q$ or WENO$p$-AB$q$, $p=3,5$, $q = 4,6,8$. The table shows that also for friction source terms, the accuracy of the underlying ODE solver is retrieved.  Compared to the standard WENO3 method, this leads to error reductions which go from 2 to 6 orders of magnitude on the finest resolutions and of 2 or 3 orders of magnitude on the intermediate one. Similar results are obtained for the subcritical stationary solution.

\begin{table}[h]
\caption{Shallow water equations with friction law $\kappa h q|q|$. Supercritical stationary solution: 
errors in $L^1$ norm for $H$ and convergence rates for WENO3GF-AM$q$, with $q=4,6,8$ and WENO$p$-nWB, $p=3,5$.
\label{tab:firct-super}}
\begin{minipage}{0.5\textwidth}
\centering\begin{tabular}{ |p{0.5cm}|c@{\hskip 5pt}c|c@{\hskip 5pt}c|c@{\hskip 5pt}c|  } 
\hline
 & \multicolumn{2}{|c|}{4 steps} & \multicolumn{2}{|c|}{6 steps} & \multicolumn{2}{|c|}{8 steps} \\
 \hline
 N & Error  & Order & Error & Order & Error & Order \\
 \hline
 \multicolumn{7}{|c|}{\textbf{WENO3GF - AM}} \\
 \hline
 20  & 3.789e-03 & -    & 2.960e-03   & -    &  2.360e-03  & -\\
 40  & 1.671e-04 & 6.3  & 9.639e-05   & 5.0  &  5.734e-05  & 5.4\\
 80 & 6.966e-06 & 4.6  & 1.225e-06   & 6.3  &  3.060e-07  & 7.6\\
 160 & 3.085e-07 & 4.5  & 1.379e-08   & 6.5  &  1.184e-09  & 8.0\\
 320 & 1.504e-08 & 4.4  & 1.842e-10   & 6.2  &  4.554e-12  & 8.0\\
 \hline
\end{tabular}
\end{minipage}\\
\centering\begin{minipage}{0.3\textwidth}
\begin{tabular}{ |p{0.5cm}|c@{\hskip 5pt}c| } 
\hline
\multicolumn{3}{|c|}{\textbf{WENO3nWB}} \\ 
\hline
 N & Error  & Order \\
 \hline
    20  & 7.606e-03 & -    \\
     40  & 1.027e-03 & 2.9    \\
 80 & 1.364e-04 & 2.9    \\
 160 & 1.743e-05 & 3.0    \\
 320 & 2.195e-06 & 3.0   \\
\hline
\end{tabular}
\end{minipage}
\begin{minipage}{0.3\textwidth}
\centering
\begin{tabular}{ |p{0.5cm}|c@{\hskip 5pt}c| } 
\hline
\multicolumn{3}{|c|}{\textbf{WENO5nWB}} \\ 
\hline
 N & Error  & Order \\
 \hline
20  &   5.706e-03  & -    \\
 40  &    3.913e-04  & 3.9  \\
     80 &    2.066e-05  & 4.3  \\
 160 &    6.632e-07  & 5.0  \\
 320 &    2.037e-08  & 5.0 \\
\hline
\end{tabular}
\end{minipage}
\end{table}

We now add the perturbation $\Delta h = 10^{-4}$ in $[0.1,0.2]$ to the supercritical stationary solution and run the simulations until the final time $T=0.08$. 
A reference solution is computed on $N=2000$ points with WENO5GF-AM8.  Figure  \ref{fig:frict-pert-sup} shows the numerical solution
obtained with WENO3GF-AM4 and WENO3-nWB using  $N=100$ points: as it can be seen, the former captures well the evolution of the perturbation while the latter gives errors of orders of magnitude higher than the initial perturbation on this relatively coarse mesh. For comparison we also report
the WENO3 solution on a much refined mesh with $N=400$ which,  of course, follows more accurately the  reference. The conclusions are similar for the numerical methods based on WENO5 and/or different ODE solvers as well as for the perturbation of the subcritical stationary solution.

\begin{figure}
\includegraphics[width=0.45\linewidth]{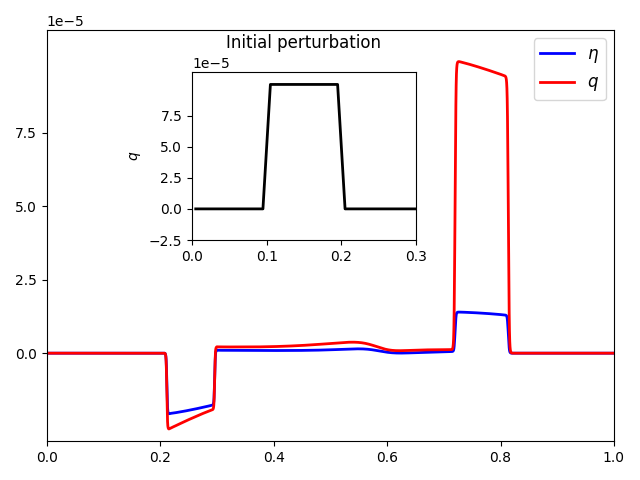}\includegraphics[width=0.55\linewidth]{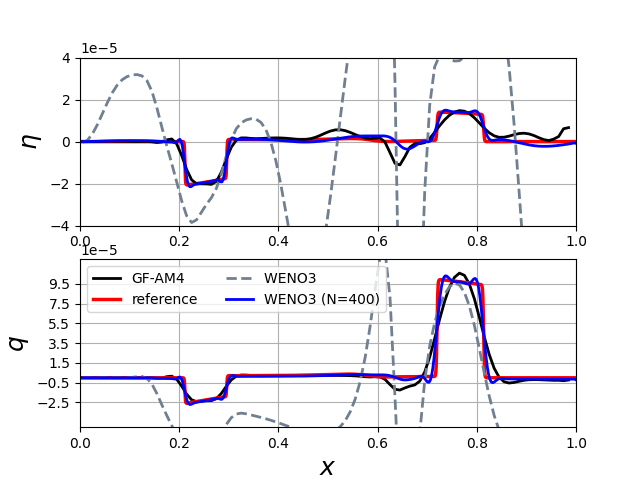}    
    \caption{Shallow water equations with friction law $\kappa h q|q|$. Small perturbation of the supercritical steady state.  Left: initial perturbation and differences between the reference and the stationary solutions. Right: differences between the stationary solution and the numerical solutions obtained with WENO3GF-AM4 using $N = 100$ points at time $t=0.08\; sec$ compared with the reference solution and the numerical solutions given by WENO3-nWB with $N = 100, 400$ point. Top: $\eta$. Down: $q$.}
    \label{fig:frict-pert-sup}
\end{figure}


\noindent
\paragraph{Moving solutions with classical Manning friction}

Last, we consider the shallow water equations with the classical Manning friction closure
$$
\kappa =  \frac{k|q|}{h^{\mu}}
$$
 where $k$ is the Manning friction coefficient and $\mu$ is set to  $ \frac{7}{3}$. 
 The Manning constant used is set to 0.05. 

As in previous cases, moving equilibria can still be achieved even when friction with a constant Manning coefficient is present. In the supercritical case, the friction term causes a reduction in physical speed from left to right, resulting in an increase in the free surface elevation across the same direction.

In this case, we first compute the discrete steady-state solution with a sweep of the  Adams-Bashforth multi-step solver (see section \ref{section:implementation_details}).  

The  solution obtained with the AB6 scheme  is reported on the left of figure \ref{fig:frict-pert-sup-manning}.   We  use  the ODE solutions   as  well prepared initial data. 
Table \ref{Table_SW_DSUP} presents the errors in $L^1$ norm for the WENO5-AB$q$ with  $q = 4,6,8$ obtained at time $t=2$ sec. As it can been seen the discrete steady solution is preserved up to machine accuracy.

\begin{table}[h]
\centering
\caption{Shallow water equations with Manning friction law. Supercritical discrete stationary solution. Errors in $L^1$ norm for WENO5GF-AB$q$, $q = 4,6,8$ at time $t=2$ sec.}
\label{Table_SW_DSUP}
\begin{tabular}{ |p{0.5cm}|c|c|c|  } 
\hline
N &  \textbf{WENO5GF-AB6} & \textbf{WENO5GF-AB8}  \\
 \hline 
 20   &   1.365e-13  &  3.912e-10\\
 40   &   2.364e-13  &  1.620e-13\\
 80   &   2.033e-13  &  3.140e-13\\
 160  &   2.516e-13  &  6.079e-13\\
 320  &   1.724e-12  &  1.481e-12\\
\hline
\end{tabular}
\end{table}

Next, we perturb the discrete steady solutions of WENO5-GF-AB6 with a small perturbation of $\Delta h= 0.0001$ m placed in the interval [7.5, 9.5] m. Figure \ref{fig:frict-pert-sup-manning} (right) compares the scheme's solution, on a supercritical flow,  with a reference solution obtained using WENO7GF-AB8, and WENO5 using $N=100$ points. The figure shows that the error of the WENO5 is of the order of the perturbation resulting to degrade the solution quality. For comparison we also plot the WENO5 solution on a refined mesh with $N=400$ points.  

A similar behavior is observed in the computation of the discrete steady solution for the subcritical case; therefore, we omit the table. In this case, the water depth is expected to decrease from left to right, while the speed increases, see figure \ref{fig:frict-pert-sub-manning} on the left. Using the discrete steady solution as well-prepared data for the initialization we add the perturbation and obtain the numerical results at $t=1$ sec.

\begin{figure}
\includegraphics[width=0.45\linewidth]{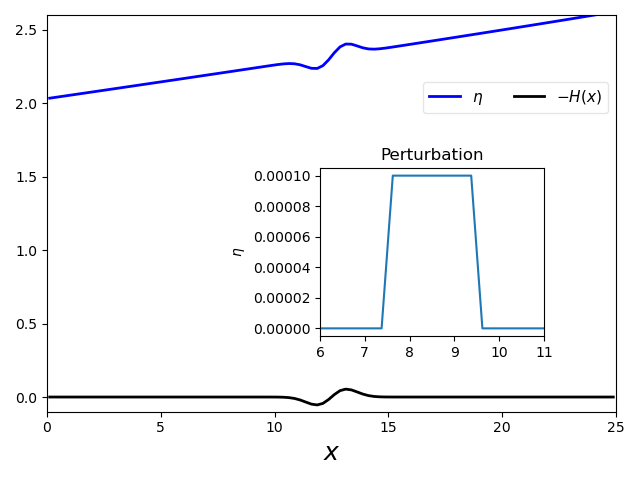}
\includegraphics[width=0.55\linewidth]{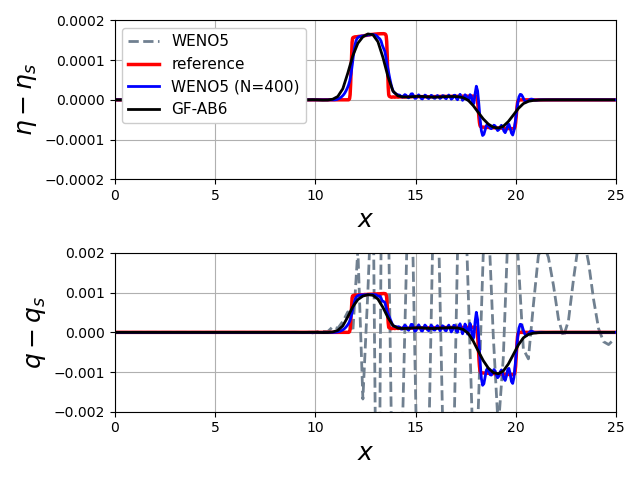}
\caption{Perturbation of the discrete steady solutions with Manning friction 0.05,  supercritical case. Left: initial discrete steady solution (Weno5-GF-AB6). Right: comparison between reference and the solutions at time $t=0.7$ obtained with the WENO5-GF-AB6 and WENO5 on $N=100$ points, and WENO5 on $N=400$ points.}
\label{fig:frict-pert-sup-manning}
\end{figure}
 
\begin{figure}
\includegraphics[width=0.45\linewidth]{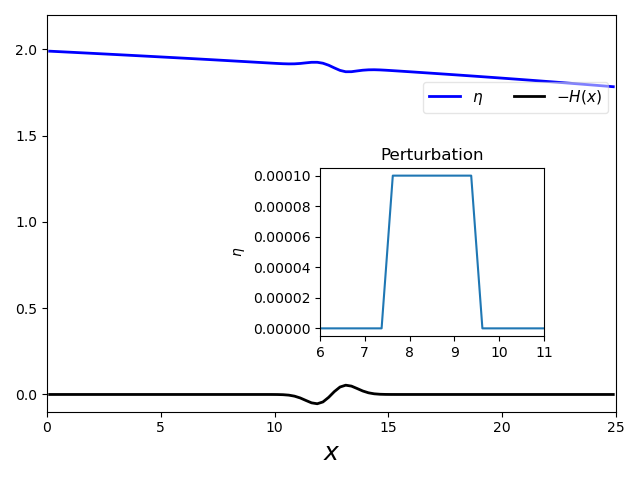}\includegraphics[width=0.55\linewidth]{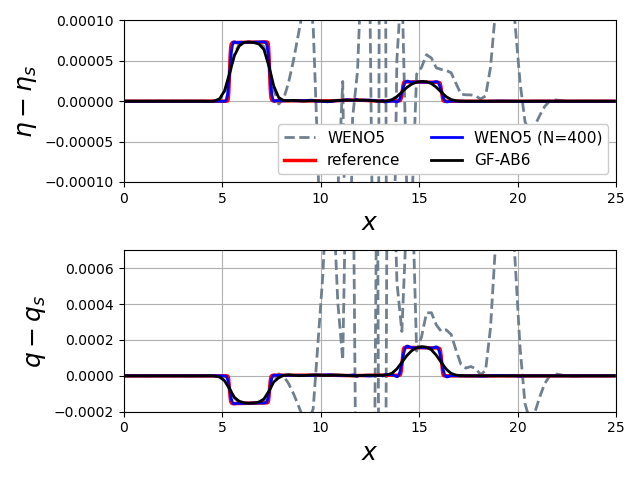}   
    \caption{Perturbation of the discrete steady solutions with Manning friction 0.05,  subcritical case. Left: initial discrete steady solution (Weno5-GF-AB6). Right: comparison between reference and the solutions at time $t=1.0$ obtained with the WENO5-GF-AB6 and WENO5 on $N=100$ points, and WENO5 on $N=400$ points.}
    \label{fig:frict-pert-sub-manning}
\end{figure}



\section{Conclusion}
In this work, we proposed high-order well-balanced methods for one-dimensional hyperbolic systems of balance laws. These methods are designed to ensure that the discrete steady states align with solutions of high-order ODE integrators, without explicitly solving the ODEs. By embedding this property directly into the scheme, we avoid applying the ODE integrator to solve the local Cauchy problem.
A sweep of the ODE solution can still be used to obtain well prepared initial data if required.

To achieve this result we have combined   a WENO finite difference framework, with
a  global flux quadrature strategy in which the source quadrature employs  weights
from   high-order multi-step ODE methods on the finite difference grid. This approach ensures a well-balanced treatment of the source term and a direct connection between the scheme’s discrete solutions and the underlying ODE integrator. In practice, we used WENO reconstructions of orders 3 to 7 with Adams-Moulton and Adams-Bashforth methods of up to order 8.

We performed classical benchmarks for two prototype equations involving moving equilibria to verify the correct implementation and convergence of our methods. For time-dependent solutions, we demonstrated that the order of accuracy is determined by the minimum between the WENO reconstruction order and the ODE integrator order. For smooth steady-state solutions, both the order of accuracy and the error values are uniquely determined by the choice of the multi-step method. This allows for significant error reduction—by several orders of magnitude on a given mesh—simply by adjusting the weights of the quadrature formula.

We also extended our methods to handle singular source terms and to preserve the water-at-rest solution for the shallow water equations. For the latter system, we proposed a strategy to find a particular family of stationary solutions, both with and without friction. 

\section*{Acknowledgements}
The work of Carlos Par\'es is partially supported by grants PID2022-137637NB-C21 and PID2022-137637NB-C22 funded by MCIN/AEI/10.13039/501100011033 and by “ERDF A way of making Europe”. Maria Kazolea and Mario Ricchiuto are members of the CARDAMOM research team of the Inria center at the University of Bordeaux.

\noindent \section*{In memoriam}
 
\noindent This paper is dedicated to the memory of Prof. Arturo Hidalgo L\'opez
($^*$July 03\textsuperscript{rd} 1966 - $\dagger$August 26\textsuperscript{th} 2024) of the Universidad Politecnica de Madrid, organizer of HONOM 2019 and active participant in many other editions of HONOM.
Our thoughts and wishes go to his wife Lourdes and his sister Mar\'ia Jes\'us, whom he left behind.

\appendix

\section{Coefficients for the AM and AB schemes}\label{app_AM-AMcoeff}
For the sake of completeness, we include here the coeffiients of the multi-step methods that are used to compute the integrals of the source term and the discrete stationary solutions. Since  we work with schemes of formal accuracy $2k+1\ge 3$
we consider ODE solvers whose solutions have accuracy at least $2k+2\ge 4$. The coefficients for the relevant methods are provided below up to order 8\footnote{for higher orders check the tables/codes here: {\tt http://www.mymathlib.com/diffeq/adams/}}.

\paragraph{AB$q$ methods} The explicit AB$q$ methods of order $q$ = 4, 6, and 8 use $q$ steps and $\beta_q = 0$. The remaining coefficients are:
\begin{equation*}
\begin{array}{lllll}
\textsf{AB4:} \\[10pt]
\beta_{3} = &\dfrac{55}{24}\\[15pt]
\beta_{2} = -&\dfrac{59}{24}\\[15pt]
\beta_{1} = &\dfrac{37}{24}\\[15pt]
\beta_{0} = -&\dfrac{9}{24}
\end{array}
\;,\quad
\begin{array}{lllllll}
\textsf{AB6:}  \\[10pt]
\beta_{5} = &\dfrac{4277}{1440}\\[15pt]
\beta_{4} = -&\dfrac{7923}{1440}\\[15pt]
\beta_{3} = &\dfrac{9982}{1440}\\[15pt]
\beta_{2} = -&\dfrac{7298}{1440}\\[15pt]
\beta_{1} = &\dfrac{2877}{1440}\\[15pt]  
\beta_{0} = -&\dfrac{475}{1440}
\end{array}
\;,\quad
\begin{array}{lllllllll}
\textsf{AB8:} \\[10pt]
\beta_{7} = &\dfrac{434241}{120960}\\[15pt]
\beta_{6} = -&\dfrac{1152169}{120960}\\[15pt]
\beta_{5} = &\dfrac{2183877}{120960}\\[15pt]
\beta_{4} = -&\dfrac{2664477}{120960}\\[15pt]
\beta_{3} = &\dfrac{2102243}{120960}\\[15pt]  
\beta_{2} = -&\dfrac{1041723}{120960}\\[15pt]    
\beta_{1} = &\dfrac{295767}{120960}\\[15pt]  
\beta_{0} = -&\dfrac{36799}{120960}
\end{array}
\end{equation*}



\paragraph{AM$q$ methods} The  AM$q$  methods of order $q$ are implicit and use $s = q-1 $ steps. The coefficients are the following:

\begin{equation*}
\begin{array}{lllll}
\textsf{AM4:} \\[10pt]
\beta_{3} = &\dfrac{9}{24}\\[15pt]
\beta_{2} = &\dfrac{19}{24}\\[15pt]
\beta_{1} = -&\dfrac{5}{24}\\[15pt]
\beta_{0} = &\dfrac{1}{24}
\end{array}\;,\quad
\begin{array}{lllll}
\textsf{AM6:} \\[10pt]
\beta_{5} = &\dfrac{475}{1440}\\[15pt]
\beta_{4} = &\dfrac{1427}{1440}\\[15pt]
\beta_{3} = -&\dfrac{798}{1440}\\[15pt]
\beta_{2} = &\dfrac{482}{1440}\\[15pt]
\beta_{1} = -&\dfrac{173}{1440}\\[15pt]  
\beta_{0} = &\dfrac{27}{1440} 
\end{array}\;,\quad
\begin{array}{lllllll}
\textsf{AM8:} \\[10pt]
\beta_{7} = &\dfrac{36799}{120960}\\[15pt]
\beta_{6} = &\dfrac{139849}{120960}\\[15pt]
\beta_{5} = -&\dfrac{121797}{120960}\\[15pt]
\beta_{3} = &\dfrac{123133}{120960}\\[15pt]
\beta_{3} = -&\dfrac{88547}{120960}\\[15pt]  
\beta_{2} = &\dfrac{41499}{120960}\\[15pt]    
\beta_{1} = -&\dfrac{11351}{120960}\\[15pt]  
\beta_{0} = &\dfrac{1375}{120960} 
\end{array}
\end{equation*}



\section{Discontinuous source linearization}\label{app_linearization}

Let us show some examples of linearizations of the source term satisfying \eqref{linearS(U)}:
\begin{itemize}
    \item  For the Burgers' equation \eqref{sbl}-\eqref{Burgers}
we can compute simple  linearizations of the source using satisfying \eqref{linearS(U)} from the relations
\eqref{ssBurgers2}, \eqref{ssBurgersp} satisfied at an admissible jump:
\begin{itemize}
\item $p=1$: $\tilde S_{i+1/2}=\bar U_{i+1/2}:=(U_i+U_{i+1})/2 $;
\item $p=2$: $\tilde S_{i+1/2}= \dfrac{\bar U_{i+1/2} [\![U]\!]_{i+1/2}}{\ln(U_{i+1}/U_i)}$;
\item $p \ge 3$: $\tilde S_{i+1/2}= (2-p)\bar U_{i+1/2} \dfrac{ [\![U]\!]_{i+1/2} }{ [\![ U^{2-p}] \!]_{i+1/2} }$
\end{itemize}
Note that all the singularities due to the division by jumps can be removed easily for $p \ge 3$. A careful implementation of the case $p=2$ is required using the $\lim_{\alpha\rightarrow 0} \alpha/\ln(1+\alpha) = 1$.

\item  For the shallow water equations two relations \eqref{ssSW} are satisfied at an admissble jump
\begin{equation*}
\begin{split}
[\![ q ]\!]_{i+1/2} = &0, \\
g[\![ H ] \!]_{i+1/2} = & g[\![ h ] \!]_{i+1/2} +  [\![ u^2/2 ] \!]_{i+1/2}.
\end{split}
\end{equation*}
Therefore one has
$$q_j = \overline{q}_{i+ 1/2}, \quad u_j = \frac{\overline{q}_{i+ 1/2}}{h_j}, \quad j = i, i+1.$$
We deduce that across an admissible jump
$$
g[\![ H ] \!]_{i+1/2}= \left\{g
-  \dfrac{\overline{q}_{i+1/2}^2 \bar h_{i+1/2}}{(h_i h_{i+1})^2}  \right\}[\![ h ] \!]_{i+1/2}.
$$
Imposing that
$$
[\![hu^2 + gh^2/2 ]\!]_{i+1/2} = g\tilde h_{i+1/2}
[\![ H ] \!]_{i+1/2}
$$
which means 
$$
-  \dfrac{\overline{q}_{i+1/2}^2 }{h_ih_{i+1}}[\![ h ] \!]_{i+1/2}   + g\bar h_{i+1/2}[\![ h ] \!]_{i+1/2}
= \tilde h_{i+1/2}  \left\{g-
  \dfrac{\overline{q}_{i+1/2}^2 \bar h_{i+1/2}}{(h_i h_{i+1})^2}  \right\}[\![ h ] \!]_{i+1/2},
$$
we end up with
$$
\tilde h_{i+1/2} =\dfrac{ g\bar h_{i+1/2}-
  \dfrac{\overline{q}_{i+1/2}^2 }{h_ih_{i+1}}   }{\left\{g
-  \dfrac{\overline{q}_{i+1/2}^2 \bar h_{i+1/2}}{(h_i h_{i+1})^2}  \right\} }=\bar  h_{i+1/2}+
 \dfrac{   \dfrac{\overline{q}_{i+1/2}^2}{g(h_i h_{i+1})^2 }(\bar h_{i+1/2}^2  -   h_ih_{i+1})    }{ 1
-   \dfrac{\overline{q}_{i+1/2}^2}{g(h_i h_{i+1})^2 } \bar h_{i+1/2} }  .
$$
Then,
$$\tilde S_{i+1/2}  = \left[ \begin{array}{c} 0 \\ - g\tilde h_{i+1/2} \end{array} \right]$$
is defined.
If the equality \eqref{intwar} is used to compute the integral of the source term, the linearization would be instead as follows
 $$
     \int_{x_i}^{x_{i+1}}ghH_x\approx g\tilde \eta_{i+1/2} [\![H]\!]_{i+1/2} -\frac{g}{2} [\![H^2]\!]_{i+1/2}
     $$
     with
     $$
     \tilde\eta_{i+1/2}:= \bar \eta_{i+1/2}
+  \dfrac{   \dfrac{\overline{q}_{i+1/2}^2}{g(h_i h_{i+1})^2 }(\bar h_{i+1/2}^2  -   h_ih_{i+1})    }{ 1
-   \dfrac{\overline{q}_{i+1/2}^2}{g(h_i h_{i+1})^2 } \bar h_{i+1/2} }  
     $$

\end{itemize}

\bibliographystyle{acm}
\bibliography{references}
 
\end{document}